 \documentclass[final,3p,times,10pt]{elsarticle}

\usepackage[dvipsnames]{xcolor}

\usepackage{graphicx}
\usepackage{epsfig}

\usepackage{amsmath}
\usepackage{amssymb}
\usepackage{amsthm}
\usepackage{amstext}

\usepackage{tikz}
\usetikzlibrary{arrows}
\usetikzlibrary{datavisualization.formats.functions}
\usetikzlibrary{matrix,calc}

\usepackage{bbm}
\usepackage{subcaption}
\usepackage{enumerate}
\usepackage{multirow}
\usepackage{natbib}
\usepackage{url}

\newtheorem{definition}{Definition}
\newtheorem{thm}{Theorem}
\newtheorem{lem}{Lemma}
\newtheorem{cor}{Corollary}
\newtheorem{ex}{Example}
\newtheorem{prop}{Proposition}
\newtheorem{rem}{Remark}

\newtheorem*{problem*}{Problem}
\newtheorem*{assumption*}{Assumption}
\newtheorem*{exs*}{Examples}
\newtheorem*{ag*}{Assumptions and Goal}
\newtheorem*{ag2*}{Assumption}

\usepackage{subcaption}

\usepackage{bbm}
\usepackage{bm}
\usepackage{pgfplots}
\pgfplotsset{compat=newest}

\pgfmathdeclarefunction{gauss}{2}{%
	\pgfmathparse{1/(#2*sqrt(2*pi))*exp(-((x-#1)^2)/(2*#2^2))}%
}

\usepackage{lipsum}
\usepackage{mathtools}
\usepackage{cancel}
\usepackage{enumerate}
\usepackage{setspace}

\usepackage{hyperref}

\usepackage{scalerel,stackengine}
\stackMath
\newcommand\reallywidehat[1]{%
	\savestack{\tmpbox}{\stretchto{%
			\scaleto{%
				\scalerel*[\widthof{\ensuremath{#1}}]{\kern-.6pt\bigwedge\kern-.6pt}%
				{\rule[-\textheight/2]{1ex}{\textheight}}
			}{\textheight}%
		}{0.75ex}}%
	\stackon[2pt]{#1}{\tmpbox}%
}

\journal{NAME OF JOURNAL}

\allowdisplaybreaks
\begin{document}

\onehalfspacing

\begin{frontmatter}

\title{Non-ergodic statistics and spectral density estimation for stationary real harmonizable symmetric $\alpha$-stable processes}

\author{Ly Viet Hoang}
\ead{ly.hoang@uni-ulm.de}
\author{Evgeny Spodarev}
\ead{evgeny.spodarev@uni-ulm.de}
\address{Ulm University}

%
%
%
%
%

\begin{abstract}
We consider non-ergodic class of stationary real harmonizable symmetric $\alpha$-stable processes $X=\left\{X(t):t\in\mathbb{R}\right\}$ with a finite symmetric and absolutely continuous control measure. We refer to its density function as the spectral density of $X$. 
These processes admit a LePage series representation and are conditionally Gaussian, which allows us to derive the non-ergodic limit of sample functions on $X$. In particular, we give an explicit expression for the non-ergodic limits of the empirical characteristic function of $X$ and the lag process $\left\{X(t+h)-X(t):t\in\mathbb{R}\right\}$ with $h>0$, respectively. 
The process admits an equivalent representation as a series of sinusoidal waves with random frequencies which are i.i.d. with the (normalized) spectral density of $X$ as their probability density function. 
Based on strongly consistent frequency estimation using the periodogram we present a strongly consistent estimator of the spectral density.
The periodogram's computation is fast and efficient, and our method is not affected by the non-ergodicity of $X$.
\end{abstract}


\end{frontmatter}



\section{Introduction}
Stationarity is often a key property in the analysis of dependence structures of time series or more generally stochastic processes. 
Especially stationary Gaussian processes have been extensively studied \cite{doob, dym, lindgren}. 
It is well known that $\alpha$-stable processes, where $\alpha\in(0,2]$ is the so-called index of stability, generalize Gaussian processes \cite{gennady}. In particular, symmetric $\alpha$-stable ($S\alpha S$) processes are of interest.  
It is the infinite second moment of $\alpha$-stable distributions (in the case $\alpha<2$) which sets them apart from Gaussian processes ($\alpha=2$), and allows them to be used for example in models of heavy-tailed random phenomena. 

It can be shown that stationary $S\alpha S$ processes fall into one of the following three subclasses -- $S\alpha S$ moving average processes, harmonizable $S\alpha S$ processes and a third class, which consists of processes characterized by a non-singular conservative flow and the corresponding cocycle \cite{rosinski}. 
The classes of moving averages and harmonizable processes are disjoint when the index of stability $\alpha$ is less then $2$. 
Only in the case $\alpha=2$, i.e. in the Gaussian case, one may find both a moving average representation and a harmonizable representation for the same process \cite[Theorem 6.7.2]{gennady}. 

Cambanis et al \cite{cambanis1} studied ergodic properties of stationary stable processes and proved that contrary to stable moving averages the harmonizable stable processes are non-ergodic .
As a consequence, the latter has garnered little attention from a statistical point of view, as classical estimation methods that rely on empirical functions are unfeasible for non-ergodic processes. 

To date, mainly the special subclass of harmonizable fractional stable motions, which are a generalization of fractional Brownian motions and belong to the class of stable self-similar processes, has been in the focus of the study of harmonizable stable processes. 
These processes play an important role in probability theory as they appear in connection with limit theorems \cite{lamperti}, as well as in stochastic modeling since the self-similarity and fractal-like behavior are features observable in real world phenomena. 
Structural, probabilistic and sample path properties have been studied in recent works \cite{charfracstable,xiao,basse,bierme} but their statistical inference tools have been missing so far. This is mainly due to the aforementioned non-ergodicity of the processes. 

Harmonizable fractional stable motions are characterized solely by their Hurst parameter and the index of stability.
The underlying stable random measure's control measure is infinite, more precisely it is the Lebesgue measure on $\mathbb{R}$. 

As a first step towards the statistical analysis of harmonizable stable processes with infinite control measure, such as the fractional stable motions, 
we consider the class of stationary harmonizable stable processes with finite control measure and integrable symmetric spectral density $f$. 
This is motivated by the fact that our case is equivalent to consider a harmonizable stable process with Lebesgue control measure and weighted kernel function $e^{itx}f^{1/\alpha}(x)$.

The class of stationary real harmonizable symmetric $S\alpha S$ (SRH $S\alpha S$) processes $X=\{X(t):t\in\mathbb{R}\}$ with index of stability $0<\alpha<2$ is defined by 
\begin{align*}
	X(t)=Re\left(\int_{\mathbb{R}}e^{itx}M(dx)\right),
\end{align*}
where $M$ is an isotropic complex $S\alpha S$ random measure with circular control measure $k$. 
This measure is a product measure on the space $(\mathbb{R}\times S^1, \mathcal{B}(\mathbb{R})\times\mathcal{B}(S^1))$ and admits the form $k=m\cdot\gamma$. The measure $m$ is called the \emph{control measure }of $M$ and $\gamma$ is the uniform probability measure on $S^1$. 
The finite dimensional distributions of the process $X$ are determined by $m$, hence the process $X$ is completely characterized by $m$ as well.
We assume that the control measure $m$ is an absolutely continuous symmetric probability measure on $\mathbb{R}$ with symmetric probability density function $f$, which we refer to as the \emph{spectral density} of the process $X$. 
It can be easily verified that the above integral representation is equivalent to 
\begin{align*}
	X(t)=Re\left(\int_{\mathbb{R}}e^{itx}f^{1/\alpha}\widetilde{M}(dx)\right),
\end{align*}
where $\widetilde{M}$ is a isotropic complex $S\alpha S$ random measure with Lebesgue control measure $m(dx)=dx$ on $\mathbb{R}$ and $f$ is a symmetric integrable function on $\mathbb{R}$.

In this paper, an approach for the estimation of the spectral density based on classical methods from spectral analysis and signal processing \cite{brockwell, hannan} is presented. 
Furthermore, we examine the asymptotic behaviour of time-averages of observables on SRH $S\alpha S$ processes and derive an ergodic theorem.  
This allows us to give an explicit expression for the non-ergodic limit of the finite-dimensional empirical characteristic function of the process $X$. 

First, Section \ref{prelim} establishes the basics on SRH $S\alpha S$ processes. 
An in-depth definition of the process as well as its properties are given. 
In particular, the non-ergodicity of SRH $S\alpha S$ processes is discussed. 
We also give a short introductory example involving $\alpha$-sine transforms on $\mathbb{R}$, and how their inversion can be used to estimate the spectral density. 

In Section \ref{nonerglimit}, the non-ergodicity of SRH $S\alpha S$ processes is examined in more detail. 
From a LePage-type series representation it follows that SRH $S\alpha S$ processes are conditionally Gaussian, albeit still non-ergodic. 
This underlying Gaussian structure can be used to study the non-ergodic behaviour of time-averages of observables of $X$.
In particular, explicit expressions for the non-ergodic limit of the finite-dimensional empirical characteristic functions of the process $X$ and the lag process $\{X(t+h)-X(t):t\in\mathbb{R}\}$ are derived. 

The fourth section dives deeper in the underlying Gaussian structure of SRH $S\alpha S$ processes. 
As a consequence of the LePage type series representation, SRH $S\alpha S$ processes can be generated by a series of sinusoidal waves with random amplitudes, phases and frequencies. 
The frequencies are independently and identically distributed with the spectral density $f$ as their probability density function. 
The periodogram is a standard tool for the estimation of frequencies in signal theory. The locations of the peaks in the periodogram are strongly consistent estimators of the aforementioned frequencies, and in
conjunction with kernel density estimators they can be used to estimate the spectral density $f$ of the SRH $S\alpha S$ process $X$.
Under minimal assumptions on the kernel function the strong consistency of the frequency estimators translates to strong or weak consistency of the kernel density estimator for $f$.

In the final section of this paper we present a thorough numerical analysis and a collection of numerical examples. 
Furthermore, some minor numerical challenges of the spectral density estimation are discussed. 
The \texttt{Matlab} and \texttt{R} implementation of our inference method can be downloaded at \cite{code}.

\section{Preliminaries}
\label{prelim}

Consider the probability space $(\Omega,\mathcal{F},P)$, and denote by $L^0(\Omega)$ the space of real-valued random variables on this probability space. 
Furthermore, define the space of complex-valued random variables on  $(\Omega,\mathcal{F},P)$ by $L^0_c(\Omega)=\left\{X=X_1+iX_2: X_1,X_2\in L^0(\Omega)\right\}.$
A real-valued random variable $X\in L^0(\Omega)$ is said to be \emph{symmetric $\alpha$-stable} if its characteristic function is of the form
\begin{align*}
	\mathbb{E}\left[e^{isX}\right]=\exp\left\{-\sigma^\alpha\vert s\vert^\alpha\right\},
\end{align*}
where $\sigma>0$ is called the \emph{scale parameter} of $X$ and $\alpha\in(0,2]$ its \emph{index of stability}. 
We write $X\sim S\alpha S(\sigma)$. 
In the multivariate case, a real-valued symmetric $\alpha$-stable random vector $\bm{X}=(X_1,\dots,X_n)$ is defined by its joint characteristic function
\begin{align*}
	\mathbb{E}\left[\exp\left\{i(\bm{s},\bm{X})\right\}\right]=\exp\left\{-\int_{S^{n-1}}\left\vert(\bm\theta,\bm{s})\right\vert^\alpha\Gamma(d\bm{\theta})\right\},
\end{align*}
where $(\bm{x},\bm{y})$ denotes the scalar product of two vectors $\bm{x},\bm{y}\in\mathbb{R}^n$, and $S^{n-1}$ is the unit sphere in $\mathbb{R}^n$. 
The measure $\Gamma$ is called the \emph{spectral measure} of $\bm{X}$. It is unique, finite and symmetric for $0<\alpha<2$ \cite[Theorem 2.4.3]{gennady}. 
A random variable $X\in L^0_c(\Omega)$ has \emph{complex symmetric $\alpha$-stable} distribution if its real and imaginary parts form a $S\alpha S$ random vector, i.e. if the vector $(Re(X),Im(X))$ is jointly $S\alpha S$. 

To give a rigorous definition of harmonizable $S\alpha S$ processes, the notion of complex random measures needs to be introduced. 
Let $(E,\mathcal{E})$ be a measurable space, and let $(S^1,\mathcal{B}(S^1))$ be the measurable space on the unit circle $S^1$ equipped with the Borel $\sigma$-algebra $\mathcal{B}(S^1)$. 
Let $k$ be a measure on the product space $(E\times S^1,\mathcal{E}\times\mathcal{B}(S^1))$, and let
$$\mathcal{E}_0=\left\{A\in\mathcal{E}:k(A\times S^1)<\infty\right\}.$$
A \emph{complex-valued $S\alpha S$ random measure} on $(E,\mathcal{E})$ is an independently scattered, $\sigma$-additive, complex-valued set function 
$$M:\mathcal{E}_0\rightarrow L^0_c(\Omega)$$ 
such that the real and imaginary part of $M(A)$, i.e. the vector $\left(Re(M(A)),Im(M(A))\right)$, is jointly $S\alpha S$ with spectral measure $k(A\times\cdot)$ for every $A\in\mathcal{E}_0$ \cite[Definition 6.1.2]{gennady}. 
We refer to $k$ as the \emph{circular control measure} of $M$, and denote by $m(A)=k(A\times S^1)$ the \emph{control measure} of $M$. 
Furthermore, $M$ is isotropic if and only if its circular control measure is of the form $$k=m\cdot\gamma,$$ where $\gamma$ is the uniform probability measure on $S^1$ \cite[Example 6.1.6]{gennady}.

Define the space of $\alpha$-integrable functions on $E$ with respect to the measure $m$ by $$L^{\alpha}(E,m)=\left\{u:E\rightarrow\mathbb{C}: \int_E\vert u(x)\vert^\alpha m(dx)<\infty\right\}.$$ 
A stochastic integral with respect to a complex $S\alpha S$ random measures $M$ is defined by
\begin{align*}
	I(u)=\int_Eu(x)M(dx),
\end{align*}
for all $u\in L^\alpha(E,m)$. 
This stochastic integration is well-defined on the space $L^\alpha(E,m)$ \cite[Chapter 6.2]{gennady}. 
In fact, for simple functions $u(x)=\sum_{j=1}^n c_j\mathbbm{1}_{A_j}(x)$, where $c_j\in\mathbb{C}$ and $A_j\in\mathcal{E}_0$, it is easily seen that the integral $I(u)$ is well-defined. 
Moreover, for any function $u\in L^\alpha(E,m)$ one can find a sequence of simple functions $\{u_n\}_{n=1}^\infty$ which converges to the function $u$ almost everywhere on $E$ with $\vert u_n(x)\vert\leq v(x)$ for some function $v\in L^\alpha(E,m)$ for all $n\in\mathbb{N}$ and $x\in E$. 
The sequence $\{I(u_n)\}_{n=1}^\infty$ converges in probability, and $I(u)$ is then defined as its limit. 

Setting $(E,\mathcal{E})=(\mathbb{R},\mathcal{B})$ and $u(t,x)=e^{itx}$, the definition of harmonizable $S\alpha S$ processes is as follows.
\begin{definition}
	The stochastic process $X=\left\{X(t):t\in\mathbb{R}\right\}$ defined by 
	\begin{align*}
		X(t)=Re\left(\int_\mathbb{R}e^{itx}M(dx)\right),
	\end{align*}
	where $M$ is a complex $S\alpha S$ random measure on $(\mathbb{R},\mathcal{B})$ with finite circular control measure $k$ (equivalently, with finite control measure $m$), is called a \emph{real harmonizable $S\alpha S$ process}. 
\end{definition}
A real harmonizable $S\alpha S$ process is stationary if and only if $M$ is isotropic, i.e. its spectral measure is of the form $k=m\cdot\gamma$ \cite[Theorem 6.5.1]{gennady}.
In this case $X$ is called a \emph{stationary real harmonizable $S\alpha S$ process}. 
Furthermore, by \cite[Proposition 6.6.3]{gennady} the finite-dimensional characteristic function of a SRH $S\alpha S$ process $X$ is given by 
\begin{align}
	\label{charfn}
	\mathbb{E}\left[\exp\left\{i\sum_{i=1}^ns_iX(t_i)\right\}\right]=\exp\left\{-\lambda_\alpha\int_\mathbb{R}\left\vert\sum_{j,k=1}^ns_js_k\cos\left(\left(t_k-t_j\right)x\right)\right\vert^{\alpha/2}m(dx)\right\}
\end{align}
with constant
$
\lambda_\alpha=\frac{1}{2\pi}\int_0^{2\pi}\vert\cos\left(x\right)\vert^\alpha dx
$
$s_1,\dots,s_n\in\mathbb{R}$ and $t_1,\dots,t_n\in\mathbb{R}$ for all $n\in\mathbb{N}$. Clearly, the process $X$ is uniquely characterized by its control measure $m$ since all its finite-dimensional distributions are determined by $m$.

We assume that the control measure $m$ is absolutely continuous with respect to the Lebesgue measure on $\mathbb{R}$ with symmetric density function $f$, which we refer to as the \emph{spectral density} of the process $X$. The goal is to estimate $f$ from one single sample of observations $\left(X(t_1),\dots,X(t_n)\right)$.

\subsection{Ergodicity and series representation of stationary real harmonizable \texorpdfstring{$S\alpha S$}{SaS} processes}
\label{sectionergodic}

In statistical physics the study of ergodic properties of random processes is motivated by the fundamental question whether long-term empirical observations of a random process evolving in time, e.g. the motion of gas molecules, suffice to estimate the mean value of the observable on the state space of the process.  
In other words, are time-averages of observables equal to their so-called phase or ensemble averages?
First groundbreaking results, i.e. the pointwise or strong ergodic theorem and mean ergodic theorem were proven by Birkhoff in 1931 and von Neumann in 1932, respectively. 
These results also sparked the general study of ergodic theory in mathematics, in particular the field of dynamical systems.

A brief introduction to ergodic theory is given in the following.
Details can be found in \cite[Chapter 10]{kallenberg}.
Let $(\Omega,\mathcal{F},\mathbb{P})$ be a probability space and $X_0:\Omega\rightarrow\mathbb{R}$ be a random variable. 
Denote by $\phi=(\phi_t)_{t\in\mathbb{R}}$ a family of transformations on $\mathbb{R}$ satisfying the semi-group property $\phi_s\phi_t=\phi_{s+t}$ for all $s,t\in\mathbb{R}$.
We say that the transform $\phi_t$ is \emph{measure-preserving} for all $t\in\mathbb{R}$ w.r.t. the measure $P_{X_0}(B)=\mathbb{P}(X_0\in B)$, $B\in\mathcal{B}(\mathbb{R})$, if 
\begin{align*}
	P_{X_0}(B)=\mathbb{P}({X_0}\in B)=\mathbb{P}({X_0}\in \phi_t^{-1}(B))=P_{X_0}(\phi_t^{-1}(B))
\end{align*}
for all $t\in\mathbb{R}$ and $B\in\mathcal{B}(\mathbb{R})$. 
Note that from the above it is easy to follow that for all $t\in\mathbb{R}$ the family of transformations $(\phi_t)_{t\in\mathbb{R}}$ and the random variable ${X_0}$ generate a continuous-time stationary process with $X(t):=\phi_tX_0\stackrel{d}{=}X_0=X(0)$. 
Ex adverso, for any stationary process $\{X(t):t\in\mathbb{R}\}$ there exist such a family $\phi=(\phi_t)_{t\in\mathbb{R}}$ and a random variable ${X_0}$ which generate the process.

Any set $I\in\mathcal{B}(\mathbb{R})$ is called \emph{invariant} with respect to $\phi$ if $\phi_t^{-1}(I)=I$ for all $t\in\mathbb{R}$. 
The family of all $\phi$-invariant sets $\mathcal{I}=\left\{I\in\mathcal{B}(\mathbb{R}):\phi_t^{-1}(I)=I ~\forall t\in\mathbb{R}\right\}$ is a $\sigma$-algebra called the $\sigma$-algebra of invariant sets.
Furthermore, denote by $\mathcal{I}_{X_0}=X_0^{-1}(\mathcal{I})=\{X_0^{-1}(I):I\in\mathcal{I}\}$  the $\sigma$-algebra of preimages of $\phi$-invariant sets generated by the random variable $X_0$.

One of the main results in Ergodic theory is Birkhoff's ergodic theorem. We will state the continuous-time version of the theorem as it can be found in \cite[Corollary 10.9]{kallenberg}. 

\begin{thm}[Continuous-time ergodic theorem]
	\label{birkhoff}
	Let $\{X(t):t\in\mathbb{R}\}$ be a continuous-time stationary process generated by the random variable $X=X(0)$ and the family of measure-preserving transformations $(\phi_t)_{t\in\mathbb{R}}$ on $(\mathbb{R},\mathcal{B}(\mathbb{R}))$ with invariant $\sigma$-algebra $\mathcal{I}$. Then, for any measurable function $h\geq0$ on $\mathbb{R}$ it holds that 
	\begin{align*}
		T^{-1}\int_0^Th(\phi_\tau X_0)d\tau\longrightarrow \mathbb{E}\left[h(X_0)\vert \mathcal{I}_{X_0}\right]\quad a.s.
	\end{align*}
	as $T\rightarrow\infty$.
\end{thm}

Note that the above result can be extended to integrable function $h$ on $\mathbb{R}$ by linearity of the integration and expectation as well as the representation $h=h_+-h_-$, where $h_+$ and $h_-$ are the positive and negative parts of $h$, respectively. 

The stochastic process $\{X(t):t\in\mathbb{R}\}$ is \emph{ergodic} if and only if $\mathbb{P}(X_0\in I)=0$ or $1$ for any $I\in\mathcal{I}$, i.e. if the $\sigma$-algebra $\mathcal{I}_{X_0}$ is $\mathbb{P}$-trivial.
In this case, the above limit becomes deterministic, i.e. the conditional expectation reduces to $\mathbb{E}\left[h(X_0)\vert \mathcal{I}_{X_0}\right]=\mathbb{E}\left[h(X_0)\right]$.

\begin{prop}
	Harmonizable $S\alpha S$ process with $0<\alpha<2$ are non-ergodic. 
\end{prop}
This result was proven by Cambanis, Hardin Jr., Weron \cite[Theorem 4]{cambanis1}. 
The authors also show that $S\alpha S$ moving averages are ergodic and Sub-Gaussian processes are non-ergodic, see \cite[Chapter 3.6, 3.7]{gennady} for the respective definitions of these processes. 

\subsection{Independent path realizations, ensemble averages and \texorpdfstring{$\alpha$}{alpha}-sine transform}
\label{multpaths}

In the case that samples of independent path realizations of a SRH $S\alpha S$ process $X$ are available, consistent estimation of a symmetric spectral density $f$ of $X$ can be performed without any problems caused by the non-ergodicity of the process.  
Consider the codifference function, which is the stable law analog of the covariance function when second moments do not exist. 
It is defined by $\tau(t)=2\Vert X(0)\Vert_\alpha^\alpha-\Vert X(t)-X(0)\Vert_\alpha^\alpha$, where $\Vert X(0)\Vert_\alpha$ and $\Vert X(t)-X(0)\Vert_\alpha$ are the scale parameters of $X(0)$ and $X(t)-X(0)$, respectively, see \cite[Chapter 2.10]{gennady}.
Using the finite-dimensional characteristic function in Equation (\ref{charfn}) these scale parameters can be computed explicitly which yields 
\begin{align}
	\label{codiff}
	\tau(t)=2\sigma^\alpha - 2^\alpha\lambda_\alpha\int_{\mathbb R}\left\vert\sin\left(\frac{tx}{2}\right)\right\vert^\alpha m(dx),
\end{align}
see \cite[Chapter, 6.7]{gennady} and \cite[Section 5.2]{viet}.
Following Equation (\ref{codiff}) and assuming the control measure $m$ to have a symmetric density function $f$ we define the $\alpha$-sine transform of $f$ as 
\begin{align}
	\label{inteq}
	T_\alpha f(t):=\int_0^\infty\left\vert\sin\left(tx\right)\right\vert^\alpha f(x)dx=\frac{2\sigma^\alpha-\tau(2t)}{2^{\alpha+1}\lambda_\alpha}=\frac{\Vert X(2t)-X(0)\Vert_\alpha^\alpha}{2^{\alpha+1}\lambda_\alpha}.
\end{align}
Consistent estimation of the right hand side of (\ref{inteq}) (based on independent realizations of $X$) and the inversion of the above integral transform, which was studied in \cite{viet}, yield an estimate of the spectral density function $f$. 

In more detail, let $X^{(1)},\dots,X^{(L)}$ denote a sample of independent realizations of the path of $X$ finitely observed at points $0=t_0<t_1<\dots<t_n<\infty$, i.e. $X^{(l)}=(X^{(l)}(t_0),\dots,X^{(l)}(t_n))$ for $l=1,\dots,L$.
Then, for any fixed time instant $t_i$ the sample $(X^{(1)}(t_i),\dots,X^{(L)}(t_i))$ consists of i.i.d. $S\alpha S(\sigma)$ random variables. 
In particular, the vector of differences $(X^{(1)}(t_i)-X^{(1)}(t_0),\dots,X^{(L)}(t_i)-X^{(L)}(t_0))$ consists of i.i.d. random samples of $X(t_i)-X(t_0)\sim S\alpha S(\sigma(t_i))$ with scale parameter $\sigma(t_i)=\Vert X(t_i)-X(t_0)\Vert_\alpha$ which depends on the lag $t_i-t_0$.

There are a handful of consistent parameter estimation techniques available in literature, e.g. McCulloch's quantile based method \cite{mcculloch} or the regression-type estimators by Koutrouvelis \cite{koutrouvelis}. 
These methods can be used to estimate the index of stability $\alpha$ and the scale parameters $\sigma(t_i)$. 
Let $(\hat\alpha, \hat\sigma(t_i))$ be a consistent estimator of $(\alpha,\sigma(t_i))$. 
Then, 
\begin{align}
	\label{Tf_est}
	\widehat T_\alpha f(t_i/2) = \frac{\hat\sigma(t_i)}{2^{\hat\alpha+1}\lambda_{\hat\alpha}}
\end{align}
is consistent for $T_\alpha f(t_i/2)$ for all $i=1,\dots,n$. 
Using the results of \cite{viet} we can estimate the Fourier transform of the spectral density $f$ at equidistant points, which then allows us to reconstruct $f$ itself using interpolation methods from sampling theory and Fourier inversion.
An illustration of the described estimation method for independent paths is given in 
Section \ref{indeppaths}, Figure \ref{fig:alphasine}. 
We will not go into further details at this point as our main interest lies within single path statistics. 

\section{Non-ergodic limit of sample functions}
\label{nonerglimit}

Trying to estimate the spectral density $f$ from a single path of a SRH $S\alpha S$ process, determining the non-ergodic almost sure limit of empirical functions is essential. 
We make use of the following LePage type series representation of $X$ which stems from the series representation of complex $S\alpha S$ random measures \cite[Section 6.4]{gennady}.  
As a consequence, SRH $S\alpha S$ processes are in fact conditionally Gaussian which allows us to use their underlying Gaussian structure for further analysis.
\begin{prop}[LePage Series representation]
	\label{series}
	Let $X=\{X(t):t\in\mathbb{R}\}$ be a SRH $S\alpha S$ process with finite control measure $m$. Then, X is conditionally stationary centered Gaussian with 
	\begin{align*}
		X(t)\stackrel{d}{=}\left(C_\alpha b_\alpha^{-1}m(\mathbb{R})\right)^{1/\alpha}G(t),\quad t\in\mathbb{R},
	\end{align*}
	where
	\begin{align}
		\label{harmlepageseries}
		G(t) = \sum_{k=1}^\infty \Gamma_k^{-1/\alpha}\left(G_k^{(1)}\cos(tZ_k)+G_k^{(2)}\sin(tZ_k)\right),\quad t\in\mathbb{R},
	\end{align}
	and the constants $b_\alpha$ and $C_\alpha$ are given by 
	\begin{align*}
		b_\alpha&=2^{\alpha/2}\Gamma\left(1+\alpha/2\right),\\
		C_\alpha&=\left(\int_0^\infty x^{-\alpha}\sin(x)dx\right)^{-1}=
		\begin{cases}
			\frac{1-\alpha}{\Gamma(2-\alpha)\cos(\pi\alpha/2)},&\alpha\neq 1,\\
			2/\pi,&\alpha=1.
		\end{cases}
	\end{align*}
	Furthermore, the sequence  $\left\{\Gamma_k\right\}_{k=1}^\infty$ denotes the arrival times of a unit rate Poisson process, $\{G_k^{(i)}\}_{k=1}^\infty$, $i=1,2$, are sequences of \ i.i.d. standard normally distributed random variables, and $\left\{Z_k\right\}_{k=1}^\infty$ is a sequence of \ i.i.d. random variables with law $m(\cdot)/m(\mathbb{R})$.
\end{prop}
\begin{ag2*}
	Additionally to the assumption that the control measure $m$ is absolutely continuous with respect to the Lebesgue measure on $\mathbb{R}$, let its density $f$ be symmetric with $m(\mathbb{R})=1$, i.e. the spectral density $f$ is in fact a probability density function.
\end{ag2*}

The process $G$ in Proposition \ref{series}, conditionally on the sequences $\left\{\Gamma_k\right\}_{k=1}^\infty$ and $\left\{Z_k\right\}_{k=1}^\infty$, is a stationary centered Gaussian process with autocovariance function
\begin{align*}
	r(t)=\mathbb{E}\left[G(0)G(t)\right]=\sum_{k=1}^\infty\Gamma_k^{-2/\alpha}\cos\left(t Z_k\right),\quad t\in\mathbb{R}
\end{align*}
\cite[Proposition 6.6.4]{gennady}. 
Note that by the Wiener-Khinchin theorem the autocovariance function $r$ and the spectral measure $\zeta$ of $G$ are directly related by the Fourier transform, i.e. $r(t)=\int_{\mathbb{R}}e^{it\omega}\zeta(d\omega)$, and one can compute
\begin{align*}
	\zeta(d\omega)=
	\sum_{k=1}^\infty\frac{\Gamma_j^{-2/\alpha}}{2}\left(\delta\left(d\omega - Z_j\right)+\delta\left(d\omega + Z_j\right)\right) 
\end{align*}
via Fourier inversion.
The spectral measure $\zeta$ is purely discrete, hence $G$ is non-ergodic, as Gaussian processes are ergodic if and only if their corresponding spectral measure is absolutely continuous \cite[Chapter 6.5]{lindgren}. 

The asymptotic behavior of time averages for non-ergodic Gaussian processes is studied in \cite{slezak}. 
There, the so-called Harmonic Gaussian processes of the form 
\begin{align}
	\label{harmonic}
	X(t)=\sum_kA_ke^{it\omega_k},
\end{align}
are considered, where $A_k$ are independent complex Gaussian random variables, and $\omega_k$ are deterministic frequencies of the process.

The autocovariance function $r$ and spectral measure $\sigma$ of the above harmonic Gaussian process are given by 
\begin{align*}
	r(t)=\sum_k\mathbb{E}\vert A_k\vert^2\cos(\omega_kt) \quad\text{and}\quad
	\sigma(d\omega)=\sum_{k}\frac{\mathbb{E}\vert A_k\vert^2}{2}\left(\delta\left(d\omega - \omega_k\right)+\delta\left(d\omega + \omega_k\right)\right)
\end{align*}
\cite[Equation (22), (23)]{slezak}. 
Since SRH $S\alpha S$ processes are conditionally Gaussian, we can embed them into the above non-ergodic harmonic Gaussian case and apply the results in \cite{slezak}.

\begin{thm}
	\label{thmamplphasefreq}
	Let $X$ be a SRH $S\alpha S$ process with finite control measure $m$.
	Then, the process $X$ admits the series representation 
	\begin{align}
		\label{amplphasefreq}
		X(t)\stackrel{d}{=}\sum_{k=1}^\infty R_k\cos\left(\Theta_k+tZ_k\right),
	\end{align}
	where $\left\{\Theta_k\right\}$ are i.i.d. uniformly distributed on $(0,2\pi)$,
	\begin{align*}
		R_k=\left(C_\alpha b_\alpha^{-1}m(\mathbb{R})\right)^{1/\alpha}\Gamma_k^{-1/\alpha}\sqrt{\left(G_k^{(1)}\right)^2+\left(G_k^{(2)}\right)^2}
	\end{align*}
	and $\{\Gamma_k\}_{k=1}^\infty,\{G^{(1)}_k\}_{k=1}^\infty,\{G^{(2)}_k\}_{k=1}^\infty,\{Z_k\}_{k=1}^\infty$ are the sequences in the series representation of $X$ from Proposition \ref{series}.
\end{thm}
\begin{proof}
	The series representation of $X$ in Equation (\ref{harmlepageseries}) clearly shows a strong resemblance to a Fourier series. Similarly to the equivalence of the sine-cosine and exponential form of Fourier series, 
	the process X admits the form
	\begin{align}
		\label{exponentialformX}
		X(t)\stackrel{d}{=}\sum_{k=-\infty}^\infty A_ke^{itZ_k},
	\end{align}
	where $\{A_k\}_{k=1}^\infty$ conditionally on $\left\{\Gamma_k\right\}_{k=1}^\infty$ is a sequence of complex Gaussian random variables  with 
	\begin{align}
		\label{Akvariancemixture}
		A_k=\left(C_\alpha b_\alpha^{-1}m(\mathbb{R})\right)^{1/\alpha}\frac{\Gamma_k^{-1/\alpha}}{2}G_k^{(1)}-i\left(C_\alpha b_\alpha^{-1}m(\mathbb{R})\right)^{1/\alpha}\frac{\Gamma_k^{-1/\alpha}}{2}G_k^{(2)}
	\end{align} 
	for $k\geq 1$, and $A_0=0$. Furthermore, we set $A_{-k}=A_k^\ast$, the complex conjugate of $A_k$, and $Z_{-k}=-Z_k$. One can easily verify that the exponential form in Equation (\ref{exponentialformX}) is indeed equal in distribution to the series representation (\ref{harmlepageseries}) of $X$ from Proposition \ref{series}. To do so, denote $\tilde{C}_k=\left(C_\alpha b_\alpha^{-1}m(\mathbb{R})\right)^{1/\alpha}\frac{\Gamma_k^{-1/\alpha}}{2}$ for ease of notation, and compute
	\begin{align*}
		A_ke^{itZ_k}&=\tilde{C}_k\left(G_k^{(1)}+iG_k^{(2)}\right)\left(\cos(tZ_k)+i\sin(tZ_k)\right)\\
		&=\tilde{C}_k\left(G_k^{(1)}\cos(tZ_k)-G_k^{(2)}\sin(tZ_k)+iG_k^{(1)}\sin(tZ_k)+iG_k^{(2)}\cos(tZ_k)\right),\\
		A_{-k}e^{-Z_{-k}}&=\tilde{C}_k\left(G_k^{(1)}-iG_k^{(2)}\right)\left(\cos(tZ_k)-i\sin(tZ_k)\right)\\
		&=\tilde{C}_k\left(G_k^{(1)}\cos(tZ_k)-G_k^{(2)}\sin(tZ_k)-iG_k^{(1)}\sin(tZ_k)-iG_k^{(2)}\cos(tZ_k)\right).
	\end{align*}
	Consequently,
	\begin{align*}
		A_ke^{itZ_k}+A_{-k}e^{-Z_{-k}}&=2\tilde{C}_k\left(G_k^{(1)}\cos(tZ_k)-G_k^{(2)}\sin(tZ_k)\right)\\&=\left(C_\alpha b_\alpha^{-1}m(\mathbb{R})\right)^{1/\alpha}\Gamma_k^{-1/\alpha}\left(G_k^{(1)}\cos(tZ_k)-G_k^{(2)}\sin(tZ_k)\right)
	\end{align*}
	is equal in distribution to the $k$-th summand of the series representation in Proposition \ref{series}, taking into account that $-G_k^{(2)}\stackrel{d}{=}G_k^{(2)}$ by symmetry of the standard normal distribution.
	
	Conditionally on $\left\{\Gamma_k\right\}_{k=1}^\infty$ and $\left\{Z_k\right\}_{k=1}^\infty$, the SRH $S\alpha S$ process X is in fact a non-ergodic harmonic Gaussian process as in Equation (\ref{harmonic}).
	Analogously to \cite[Equation (27)]{slezak}, setting
	\begin{align*}
		R_k&=2\vert A_k\vert
		=\left(C_\alpha b_\alpha^{-1}m(\mathbb{R})\right)^{1/\alpha}\Gamma_k^{-1/\alpha}\sqrt{\left(G_k^{(1)}\right)^2+\left(G_k^{(2)}\right)^2}
	\end{align*}
	yields an alternative series representation for the process $X$ with
	\begin{align*}
		X(t)\stackrel{d}{=}\sum_{k=1}^\infty R_k\cos\left(\Theta_k+tZ_k\right),
	\end{align*}
	where $R_k$ are the amplitudes of the spectral points at the frequencies $Z_k$, and $\Theta_k$ are i.i.d. uniformly distributed phases of $A_k$ on $(0,2\pi)$. 
\end{proof}

It is now possible to state the non-ergodic almost sure limit of the empirical characteristic function of an SRH $S\alpha S$ process $X=\{X(t):t\in\mathbb{R}\}$.
\begin{thm}
	\label{nonerglim}
	
	Let $X$ be a SRH $S\alpha S$ process with finite control measure $m$. Then, for all $\lambda\in\mathbb{R}$
	\begin{align*}
		T^{-1}\int_0^Te^{i\lambda X(\tau)}d\tau\longrightarrow\mathbb{E}\left[e^{i\lambda X(t)}\vert\mathcal{A}\right]=\prod_{k=1}^\infty J_0(\lambda R_k)\quad a.s.
	\end{align*}
	as $T\rightarrow\infty$ , where $\mathcal{A}=\sigma(\{\Gamma_k^{-1/\alpha}G_k^{(1)},\Gamma_k^{-1/\alpha}G_k^{(2)}, Z_k\}_{k\in\mathbb{N}})$ and $J_0(s)=\frac{1}{2\pi}\int_{-\pi}^\pi e^{is\cos(u)}du$ is the Bessel function of the first kind of order $0$.
\end{thm}
\begin{proof}	
	Using the series representation in Equation (\ref{amplphasefreq}) from Theorem \ref{thmamplphasefreq}, we first compute the characteristic function of $X(t)$ conditional on $\mathcal{A}$. 
	Note that the amplitudes $\{R_k\}$ are measurable functions of $\Gamma_k^{-1/\alpha}G_k^{(1)},\Gamma_k^{-1/\alpha}G_k^{(2)}$, and are  therefore measurable with respect to $\mathcal{A}$.
	For any $\lambda\in\mathbb{R}$ it holds that 
	\begin{align*}
		\mathbb{E}\left[\exp\left(i\lambda X(t)\right)\vert \mathcal{A}\right]=&\mathbb{E}\left[\exp\left(i\lambda\sum_{k=1}^\infty R_k\cos\left(\Theta_k + t Z_k\right)\right)\vert \mathcal{A} \right]
		=\mathbb{E}\left[\prod_{k=1}^\infty\exp\left(i\lambda R_k\cos\left(\Theta_k + t Z_k\right)\right)\vert \mathcal{A} \right]\\
		=&\prod_{k=1}^\infty\underbrace{\mathbb{E}\left[\exp\left(i\lambda R_k\cos\left(\Theta_k + t Z_k\right)\right)\vert \mathcal{A} \right]}_{\frac{1}{2\pi}\int_{-\pi}^\pi\exp\left(i\lambda R_k\cos(u)\right)du=J_0\left(\lambda R_k\right)}=\prod_{k=1}^\infty J_0(\lambda R_k).
	\end{align*}
	
	Birkhoff's ergodic theorem, see Theorem \ref{birkhoff} in Section \ref{sectionergodic}, readily states that the empirical characteristic function $T^{-1}\int_0^Te^{i\lambda X(\tau)}d\tau$ converges almost surely to the conditional expectation $\mathbb{E}\left[e^{i\lambda X(t)}\vert \mathcal{I}_{X_0}\right]$. Note that the theorem is stated for all measurable functions $h\geq 0$ but can easily be extended to our case. Linearity of integration and expectation allows us to consider the real and imaginary part of the complex exponential separately. Each can then be expressed as the difference of positive and negative parts. 
	
	What remains to be shown is the equality $\mathcal{I}_{X_0}=\mathcal{A}$. 
	The inclusion $\mathcal{I}_{X_0}\subseteq\mathcal{A}$ follows easily from the LePage-type series representation in Proposition \ref{series}. It holds that $X(t)$ is a series of measurable functions of $\Gamma_k^{-1/\alpha}G_k^{(1)}$, $\Gamma_k^{-1/\alpha}G_k^{(2)}$ and $Z_k$, $k\in\mathbb{N}$, see Equation (\ref{harmlepageseries}). Hence, it is $\mathcal{A}$-measurable. Additionally, the inclusion $\mathcal{I}\subseteq\mathcal{B}(\mathbb{R})$ holds, which yields
	$
	\mathcal{I}_{X_0}=X^{-1}(\mathcal{I})\subseteq X^{-1}(\mathcal{B}(\mathbb{R}))\subseteq\mathcal{A}.
	$
	
	For the reverse inclusion $\mathcal{I}_{X_0}\supseteq\mathcal{A}$ note that $\mathcal{A}=\sigma(\mathcal{E})$, where $\mathcal{E}$ is a system of sets that generates $\mathcal{A}$. 
	A suitable choice of $\mathcal{E}$ is given by the collection of cylinder sets of the form
	\begin{align*}
		A=\left\{\omega\in\Omega:\Gamma_k^{-1/\alpha}(\omega)G_k^{(1)}(\omega)\leq a_k, \Gamma_k^{-1/\alpha}(\omega)G_k^{(2)}(\omega)\leq b_k, Z_k(\omega)\leq c_k, k=1,\dots,n\right\}
	\end{align*} 
	with $a_k,b_k,c_k\in\mathbb{R}$, $k=1,\dots,n$, and $n\in\mathbb{N}$. 
	
	Consider the set $B=\{X(s;\omega):s\in\mathbb{R},\omega\in A\}$. 
	For any $x\in B$ there exist $s\in\mathbb{R}$ and $\omega\in A$ such that $x=X(s;\omega)$. Then, $\phi_tx=X(s+t;\omega)\in B$. Conversely, for $x\in\phi_t^{-1}(B)$, it holds that $T_tx\in B$, i.e. there exist $s\in\mathbb{R}$ and $w\in A$ such that $T_tx=X(s;\omega)$. Clearly, $x=\phi_{-t}\phi_tx=X(s-t;\omega)\in B$. 
	Hence, the set $B$ is $\phi$-invariant as it holds that $\phi_t^{-1}(B)=B$ for all $t\in\mathbb{R}$, i.e. $B\in\mathcal{I}$. 
	
	By definition of $\mathcal{I}_{X_0}$ we have $X_0^{-1}(B)\in\mathcal{I}_{X_0}$, and as a consequence it holds that $\sigma(X_0^{-1}(B))\subseteq \mathcal{I}_{X_0}$. 
	Denote by $X_0(A)$ the set $\{X(0;\omega):\omega\in A\}$, which clearly is a subset of $B$. It follows that $A=X_0^{-1}(X_0(A))\subseteq X_0^{-1}(B)$. Since the $\sigma$-algebra $\mathcal{A}$ is generated by all sets $A\in\mathcal{E}$ we conclude that $\mathcal{A}\subseteq\sigma(X_0^{-1}(B))\subseteq\mathcal{I}_{X_0}$.
\end{proof}

\begin{rem}
	\label{lagprocess}
	For the lag process $\left\{X(t+h)-X(t):t\in\mathbb{R}\right\}$ with $h>0$ it holds that
	\begin{align*}
		\lim_{T\rightarrow\infty}\frac{1}{T}\int_0^Te^{i\lambda X(\tau+h)-X(\tau)}d\tau=\prod_{k=1}^\infty J_0\left(2\lambda R_k\sin\left(hZ_K/2\right)\right)\quad a.s.
	\end{align*}
	For the proof note that the computation of $\mathbb{E}\left[\exp\left(i\lambda(X(t+h)-X(t))\right)\vert\left\{R_k\right\},\left\{Z_k\right\}\right]$ is straightforward using $\cos(\Theta_k+(t+h)Z_k)-\cos(\Theta_k+t Z_k)=-2\sin\left(hZ_k/2\right)\sin\left(\Theta_k+t Z_k+hZ_k/2\right)$, where the latter term $-\sin\left(\Theta_k+t Z_k+hZ_k/2\right)$ is equal in distribution to $\cos\left(\Theta_k\right)$ with $\Theta_k\sim U(0,2\pi)$ i.i.d. 
\end{rem}

\begin{rem}
	Note that $\prod_{k=1}^\infty J_0(\lambda R_k)\not\equiv 0$ on $\mathbb{R}$. 
	Clearly, it holds that $\prod_{k=1}^\infty J_0(\lambda R_k)=1$ for $\lambda=0$ since $J_0(0)=1$. This makes sense since the empirical characteristic function $T^{-1}\int_0^Te^{i\lambda X(\tau)}d\tau=1$ for $\lambda=0$. 
	
	Let $\lambda>0$ and denote by $R_{[k]}$ be the $k$-th largest amplitude, i.e. $R_{[1]}\geq R_{[2]}\geq\dots$ and by $j_{0,1}\approx 2.4048$ the first positive root of the Bessel function $J_0$.
	Then, by monotonicity it holds that $0<J_0(\lambda R_{[k]})\leq J_0(\lambda R_{[k+1]})<1$ for all $\lambda\in(0,j_{0,1}/R_{[1]})$ and $k\in\mathbb{N}$.
	
	Consider $\lambda\in(0,2/R_{[1]})$. Applying the logarithm to the infinite product yields a series of logarithms which converges absolutely, as one can bound
	\begin{align*}
		\sum_{k=1}^\infty\left\vert\log\left(J_0\left(\lambda R_{k}\right)\right)\right\vert
		&\leq\frac{\lambda^2}{4-\lambda^2R_{[1]}}\sum_{k=1}^\infty R_{k}^2\\
		&=\frac{\lambda^2}{4-\lambda^2R_{[1]}}\left(C_\alpha b_\alpha^{-1}m(\mathbb{R})\right)^{2/\alpha}\sum_{k=1}^\infty\Gamma_k^{-2/\alpha}\underbrace{\left(\left(G_k^{(1)}\right)^2+\left(G_k^{(1)}\right)^2\right)}_{=W_k^2}.
	\end{align*}
	This upper bound holds true by the monotonicity of $J_0(\lambda R_{[k]})$ in $k\in\mathbbm{N}$ for $\lambda\in(0,j_{0,1}/R_{[1]})$ and the inequalities $\log(1/x)\leq 1/x-1$ for $x>0$ as well as $J_0(x)\geq 1-x^2/4$ for $x\in(0,2)$. The latter follows from the Taylor series expansion of $J_0$. 
	
	The series $\sum_{k=1}^\infty\Gamma_k^{-2/\alpha}W^2_k$ converges almost surely for all $\alpha\in(0,2)$ by \cite[Theorem 1.4.5]{gennady}, since $W_k^2$ are i.i.d. $\chi^2_2$-distributed with finite moments $\mathbb{E}\vert W^2_k\vert^{\alpha/2}=\mathbbm{E}\vert W_k\vert^\alpha<\infty$ for $\alpha\in(0,2)$. Ultimately, it follows that $\prod_{k=1}^\infty J_0(\lambda R_k)\in(0,1)$ almost surely for $\lambda\in(0,2/R_{[1]})$.
\end{rem}
\begin{cor}
	Let $h:\mathbb{R}\rightarrow\mathbb{R}$ be an integrable function with integrable Fourier transform $\mathcal{F}h$ such that $h(x)=\int_{\mathbb{R}}\mathcal{F}h(y)e^{ixy}dy$. Then 
	\begin{align*}
		\lim_{T\rightarrow\infty}\frac{1}{T}\int_0^Th(X(\tau))d\tau=\int_\mathbb{R}\mathcal{F}h(y)\left(\prod_{k=1}^\infty J_0(yR_k)\right)dy\quad a.s.
	\end{align*}
	If $h$ is $p$-periodic with absolutely convergent Fourier series $h(x)=\sum_{n\in\mathbb{Z}}c_ne^{i\frac{p}{2\pi}nx}$, then
	\begin{align*}
		\lim_{T\rightarrow\infty}\frac{1}{T}\int_0^Th(X(\tau))d\tau=\sum_{n\in\mathbb{Z}}c_n\prod_{k=1}^\infty J_0\left(n\frac{p}{2\pi}R_k\right)\quad a.s.
	\end{align*}
\end{cor}
\begin{proof}
	We prove the first part only since all arguments are applicable for the periodic $h$ as well. By Fubini's theorem and Lebesgue's dominated convergence theorem we can interchange the order of integration and limit, which yields
	\begin{align*}
		\lim_{T\rightarrow\infty}\frac{1}{T}\int_0^Th(X(\tau))d\tau&=\lim_{T\rightarrow\infty}\frac{1}{T}\int_0^T\int_{\mathbb{R}}\mathcal{F}h(y)\exp\left\{iy\sum_{k=1}^\infty R_k\cos\left(\Theta_k+\tau Z_k\right)\right\}dyd\tau\\
		&=\int_{\mathbb{R}}\mathcal{F}h(y)\left(\lim_{T\rightarrow\infty}\frac{1}{T}\int_0^T\exp\left\{iy\sum_{k=1}^\infty R_k\cos\left(\Theta_k+\tau Z_k\right)\right\}d\tau\right)dy\\
		&=\int_\mathbb{R}\mathcal{F}h(y)\left(\prod_{k=1}^\infty J_0(yR_k)\right)dy\quad a.s.
	\end{align*}
\end{proof}

Summarizing, it becomes clear that, conditionally on $\left\{R_k\right\}_{k=1}^\infty$ and $\left\{Z_k\right\}_{k=1}^\infty$, SRH $S\alpha S$ processes are non-ergodic harmonic Gaussian processes from Equation (\ref{harmonic}).
For SRH $S\alpha S$ processes additional randomness is introduced by the random frequencies $\{Z_k\}_{k=1}^\infty$ and the arrival times $\{\Gamma_k\}_{k=1}^\infty$, which are in a sense random variances to the Gaussian random variables $\{G_k^{(i)}\}_{k=1}^\infty$, $i=1,2$. 
Indeed, the sequence $\left\{A_k\right\}_{k=1}^\infty$ in Equation (\ref{Akvariancemixture}) consists of variance mixtures of Gaussian random variables.

\section{Spectral density estimation}

The underlying Gaussian structure of the SRH $S\alpha S$ process $X$ plays a central role in the estimation of the spectral density. 
The LePage series representation and Theorem \ref{thmamplphasefreq} demonstrate that the randomness of the paths of $X$ is generated by the Gaussian variance mixture amplitudes $\left\{R_k\right\}_{k=1}^\infty$ as well as the random frequencies $\left\{Z_k\right\}_{k=1}^\infty$. 
Although these quantities are inherently random and not directly observable, they are fixed for a given path. 
It is therefore possible to consider a path of the harmonizable process $X$ as a signal generated by frequencies $\left\{Z_k\right\}_{k=1}^\infty$ with amplitudes $\left\{R_k\right\}_{k=1}^\infty$ and phases $\left\{U_k\right\}_{k=1}^\infty$ (see Equation (\ref{amplphasefreq})), and use standard frequency estimation techniques to estimate $\left\{Z_k\right\}_{k=1}^\infty$.

Recall that  $Z_1,Z_2,\dots$ are i.i.d. with probability density function $f$, which is the spectral density function of the process $X$. A density kernel estimate will then yield the desired estimate of $f$. 
As the goal is to estimate the spectral density $f$ alone, we can neglect the estimation of amplitudes $\left\{R_k\right\}_{k=1}^\infty$ and phases $\left\{U_k\right\}_{k=1}^\infty$. 

\subsection{Periodogram method}
Spectral density estimation is a widely studied subject in the field of signal processing, and we will rely on classical techniques provided there \cite{brockwell,fuller,hannan}.
In particular, the periodogram is a standard estimate of the density function $f_\zeta$ of an absolutely continuous spectral measure $\zeta$, often also referred to as power spectral density.
When the spectral measure of a signal is purely discrete, i.e. the signal is produced by sinusoidal waves alone, the periodogram still proves to be a powerful tool for the estimation of frequencies \cite{hannan}. 
Such a sinusoidal signal $X=\left\{X(t):t\in\mathbb{R}\right\}$ is modeled by  
\begin{align}
	\label{sinusoidal}
	X(t)=\mu+\sum\limits_{k=1}^\infty r_k\cos\left(\phi_k+\lambda_kt\right)+Y(t),
\end{align}
where $\mu$ is called the overall mean of the signal and $r_k, \phi_k, \lambda_k$ are the amplitude, phase and frequency of the $k$-th sinusoidal. 
The process $Y=\left\{Y(t):t\in\mathbb{R}\right\}$ is a noise process usually assumed to be a stationary zero mean process with spectral density $f_y$ \cite[p.5]{hannan}.

Note that in the above model the amplitudes, phases and frequencies are deterministic and randomness is introduced by the noise process $Y$. 
For harmonic Gaussian processes it is assumed that the phases are i.i.d. uniformly distributed on $(0,2\pi)$ and the amplitudes are given by $R_k=\vert A_k\vert$, where $A_k$ are complex Gaussian random variables as seen in \cite{slezak}. 
For SRH $S\alpha S$ process these Gaussian random variables are replaced by variance mixtures of Gaussian variables, and, additionally, the frequencies are i.i.d. random variables with density $f$. 
The amplitudes $\left\{R_k\right\}_{k=1}^\infty$ and frequencies $\left\{Z_k\right\}_{k=1}^\infty$ in (\ref{amplphasefreq}) are fixed for a given path of the SRH $S\alpha S$ process $X$. Furthermore, we have $\mu=0$. 
Also, the phases $\left\{U_k\right\}_{k=1}^\infty$ do not play any role in the following computation of the periodogram. 

We define the discrete Fourier transform and the periodogram as follows.
\begin{definition}
	Let $x=\left(x(1),\dots,x(n)\right)\in\mathbb{C}^n$ be a (complex) vector. Then, the transform
	\begin{align*}
		F_n\left(\theta_j\right)=n^{-1}\sum\limits_{k=1}^nx(k)e^{-ik\theta_j}
	\end{align*}
	defines the \emph{discrete Fourier transform} of $x$ at the Fourier frequencies $\{\theta_j\}_{j=1}^n=\left\{2\pi j/n\right\}_{j=1}^n$.
	Furthermore, the \emph{periodogram} of $x$ at the Fourier frequencies $\{\theta_j\}_{j=1}^n$ is defined by 
	\begin{align*}
		I_n\left(\theta_j\right)=\left\vert F_n\left(\theta_j\right)\right\vert^2 = \left\vert n^{-1}\sum\limits_{k=1}^nx(k)e^{-ik\theta_j}\right\vert^2.
	\end{align*}
\end{definition}
Since the discrete Fourier transform can be viewed as a discretization of the continuous time Fourier transform, the definition of the periodogram can simply be extended to arbitrary frequencies by replacing the Fourier frequencies $\theta_j$ with $\theta>0$ in the above. 
The periodogram is a standard tool for the estimation of the spectral density if the spectral measure is absolutely continuous with respect to the Lebesgue measure on $\mathbb{R}$. 
In the case that the spectral measure is purely discrete it can be used to estimate the underlying frequencies of a signal in the following way. 

Let $\left(x(1),\dots x(n)\right)=\left(X(t_1),\dots,X(t_n)\right)$ be a sample generated from the SRH $S\alpha S$ process $X=\left\{X(t):t\in\mathbb{R}\right\}$ at equidistant points $t_j=j\delta$, $\delta>0$, for $j=1,\dots,n$. 
Then, according to Equation (\ref{amplphasefreq}) the sample $x(j)$ is of the form
\begin{align*}
	x(j)=\sum\limits_{l=1}^\infty R_l\cos\left(\Theta_l+j\delta Z_l\right)
\end{align*}
for all $j=1,\dots,n$.
For $\theta>0$, the discrete Fourier transform is given by 
\begin{align}
	\label{periodoidea}
	F_n(\delta\theta)
	&=n^{-1}\sum\limits_{j=1}^nx\left(j\right)e^{-ij\theta\delta}=n^{-1}\sum\limits_{j=1}^n\sum\limits_{l=1}^\infty R_l\cos\left(\Theta_l+j\delta Z_l\right)e^{-ij\theta\delta}\\
	&=n^{-1}\sum\limits_{l=1}^\infty R_l \sum\limits_{j=1}^ne^{-ij\theta\delta}\underbrace{\cos\left(\Theta_l+j\delta Z_l\right)}_{=\frac{e^{i\left(\Theta_l+j\delta Z_l\right)}+e^{-i\left(\Theta_l+j\delta Z_l\right)}}{2}}\nonumber\\
	&
	=n^{-1}\sum\limits_{l=1}^\infty\Bigg( \frac{e^{i\Theta_l}R_l}{2}\sum\limits_{j=1}^ne^{ij\delta\left(Z_l-\theta\right)}
	+\frac{e^{-i\Theta_l}R_l}{2}\sum\limits_{j=1}^ne^{-ij\delta\left(Z_l+\theta\right)}\Bigg).
\end{align}
Taking the squared modulus of the right-hand side expression in Equation (\ref{periodoidea}) yields the periodogram. For the full computation see \ref{appendixA}, and for more details we refer to \cite{hannan}. 

For any fixed $k\in\mathbb{N}$ the periodogram behaves like
\begin{align}
	\label{periodocomp}
	I_n(\delta\theta)
	&=\frac{R_k^2}{4}\Biggl(\underbrace{\frac{\sin^2\left(\frac{n}{2}\delta( Z_k -\theta)\right)}{n^2\sin^2\left(\frac{1}{2}\delta( Z_k -\theta)\right)}}_{(\ast)}
	+\underbrace{\frac{\sin^2\left(\frac{n}{2}\delta( Z_k+\theta)\right)}{n^2\sin^2\left(\frac{1}{2}\delta( Z_k +\theta)\right)}}_{(\ast\ast)}
	\Biggr) + O\left(n^{-2}\right)
\end{align}
as $\theta\rightarrow\vert Z_k\vert$ and $n\rightarrow\infty$.
Analogously to \cite[pp. 35-36]{hannan}, for $\theta>0$ and $Z_k>0$ the second term $(\ast\ast)$ vanishes with rate $O(n^{-2})$ as $n\rightarrow\infty$. Furthermore, the first term $(\ast)$ converges to $1$ as $\theta\rightarrow Z_k$ or to $0$ for all $\theta\not\in\{Z_k\}_{k\in\mathbb{N}}$ as $n\rightarrow\infty$ with rate $O(n^{-2})$  \cite[p. 515]{hannan2}.
On the other hand, for $\theta>0$ and $Z_k<0$ the first term $(\ast)$ in Equation (\ref{periodocomp}) behaves like $O(n^{-2})$ as $n\rightarrow\infty$. Again, the second term $(\ast\ast)$ converges to $1$ as $\theta\rightarrow  Z_k$ and to $0$ for all $\theta\not\in\{Z_k\}_{k\in\mathbb{N}}$ as $n\rightarrow\infty$. 
Therefore, the periodogram $I_n$ shows pronounced peaks close to the absolute values of the true frequencies $\left\{Z_k\right\}$.
The height of the peak at a frequency $Z_k$ is given by $R_k^2/4$. 

Let $R_{[k]}$ be the $k$-th largest amplitude, i.e. $R_{[1]}> R_{[2]}>\dots>0$ a.s., and denote by $Z_{[k]}$ the frequency associated to $R_{[k]}$. 
Taking the location of the $N\in\mathbb{N}$ largest peaks of $I_n$ as estimators for $\{\vert Z_{[k]}\vert\}_{k=1}^N$ might be intuitive but is not feasible, as there are constraints on the minimal distance between the $\vert Z_{[k]}\vert$, see \cite[Equation (5.5) and (5.6)]{walker}. 
Instead, an iterative approach is proposed, see e.g. \cite[Chapter 3.2]{hannan} and \cite[Section 5]{walker}.

Set coefficients $(a_k,b_k)_{k\in\mathbb{N}}$ with
\begin{align*}
	a_k=\left(C_\alpha b_\alpha^{-1}m(\mathbb{R})\right)^{1/\alpha}\Gamma_k^{-1/\alpha}G_k^{(1)},\quad
	b_k=\left(C_\alpha b_\alpha^{-1}m(\mathbb{R})\right)^{1/\alpha}\Gamma_k^{-1/\alpha}G_k^{(2)}
\end{align*}
such that the series representation of X from Proposition \ref{series} is given by $$X(t)=\sum_{k=1}^\infty a_k\cos\left(Z_kt\right)+b_k\sin\left(Z_kt\right).$$
Also, note that $R_k=(a_k^2+b_k^2)^{1/2}$.
We denote by the subscript $[k]$ in $a_{[k]}$ and $b_{[k]}$ their association to $R_{[k]}$. 

Again, consider the sample $(x(1),\dots,x(n))=(X(t_1),\dots,X(t_n))$,
and denote by $I_n=I_n^{(0)}$ the periodogram computed from the sample $x(1),\dots,x(n)$. Determine the estimate $\hat{Z}_{1,n}$ for $\vert Z_{[1]}\vert$ from the location of the largest peak of $I_n^{(0)}$. 
This peak is asymptotically unique, since the amplitudes satisfy $R_{[1]}>R_{[2]}>\dots$ a.s.
Compute estimators $(\hat{a}_{1,n},\hat{b}_{1,n})$ for $(a_{[1]},b_{[1]})$ by regressing $a_{[1]}\cos(t\hat{Z}_{1,n})+b_{[1]}\sin(t\hat{Z}_1)$, $t=1,\dots,n$, on $x(1),\dots,x(n)$. 

Next, compute the periodogram $I_n^{(1)}$ of the residuals $x^{(1)}(t)=x(t)-\hat{a}_{1,n}\cos(t\hat{Z}_{1,n})-\hat{b}_{1,n}\sin(t\hat{Z}_1)$, $t=1,\dots,n$. Determining the location of the largest peak of $I_n^{(1)}$ yields the estimate $\hat{Z}_{2,n}$ for $\vert Z_{[2]}\vert$. As in the first step, compute the regression estimates $(\hat{a}_{[2],n},\hat{b}_{[2],n})$ for $(a_{[2]},b_{[2]})$ and the periodogram $I_n^{(2)}$ from $x^{(2)}(t)=x^{(1)}(t)-\hat{a}_2\cos(t\hat{Z}_{2,n})-\hat{b}_2\sin(t\hat{Z_2})$, $t=1,\dots,n$.  
We repeat this process for fixed $N\in\mathbb{N}$ iterations. 
The choice of $N$ (depending on the sample size $n$) is discussed in Section \ref{numericsperiodogram}.

\begin{thm}
	\label{strongconsistencyZk}
	Let $X=\{X(t):t\in\mathbb{R}\}$ be a SRH $S\alpha S$ process with finite control measure $m$ and symmetric spectral density $f$.
	Furthermore, let $\left(x(1),\dots,x(n)\right)=(X(t_1),\dots,X(t_n))$ be a sample of the process $X$ at equidistant points $t_1<\dots<t_n$, $n\in\mathbb{N}$ with $t_j=j\delta$, $\delta>0$ for $j=1,\dots,n$.
	Denote by $\{\hat{Z}_{k,n}\}_{k=1}^N$ the periodogram frequency estimators for $\{Z_{[k]}\}_{k=1}^N$ as well as by $\{\hat{a}_{k,n}\}_{k=1}^N$ and $\{\hat{b}_{k,n}\}_{k=1}^N$ the regression estimators for $\{{a}_{[k]}\}_{k=1}^N$ and $\{{b}_{[k]}\}_{k=1}^N$, as described above. Then, it holds that 
	\begin{align*}
		\lim_{n\rightarrow\infty} n\left(\vert Z_{[k]}\vert-\hat{Z}_{k,n}\right)=0 \  , \quad 
		\lim_{n\rightarrow\infty}\hat{a}_{k,n}=a_{[k]} \  , \quad
		\lim_{n\rightarrow\infty}\hat{b}_{k,n}=b_{[k]} \quad a.s.
	\end{align*}
	for $k=1,\dots,N$.
\end{thm}
\begin{proof}
	Recall that $\left(\Omega,\mathcal{F},P\right)$ denotes the probability space on which the process $X$ lives. For any $\omega\in\Omega$ we denote the corresponding path of $X$ by
	$X(\omega)=\left\{X(t;\omega):t\in\mathbb{R}\right\}$ and periodogram by $I_n(\theta;\omega)$. 
	From Equation (\ref{periodocomp}) we know that the periodogram $I_n(\delta\theta;\omega)$ converges to $R^2_{[k]}(\omega)/4$ as $\theta\rightarrow\vert Z_{[k]}(\omega)\vert$ and $n\rightarrow\infty$. On the other hand, the periodogram converges to 0 as $n\rightarrow\infty$ for all $\theta\not\in\left\{\vert Z_{[k]}\vert \right\}_{k\in\mathbb{N}}$. For simplicity we can assume $\delta=1$ in the following.
	
	Similar to \cite{hannan, walker} one can first show that $\hat{Z}_{1,n}\rightarrow\vert Z_{[1]}\vert$ a.s. as $n\rightarrow\infty$. To see this, assume that $\hat{Z}_{1,n}(\omega)$ does not converge to $\vert Z_{[1]}(\omega)\vert$ but instead to some $z'\neq \vert Z_{[1]}(\omega)\vert$. We can distinguish between the cases $z'\not\in\{Z_{[k]}(\omega)\}_{k\in\mathbb{N}}$  and $z'=\vert Z_{[k']}(\omega)\vert$ for some $k'>1$. Then, for $n\rightarrow\infty$
	\begin{align*}
		\frac{R^2_{[1]}(\omega)}{4}\stackrel{}{\leftarrow} I_n(\vert Z_{[1]}(\omega)\vert;\omega)\leq  I_n(\hat{Z}_{1,n}(\omega);\omega)\rightarrow
		\begin{cases}
			0 \ ,&z'\not\in\{\vert Z_{[k]}(\omega)\vert\}_{k\in\mathbb{N}},\\
			\frac{R_{[k']}^2(\omega)}{4} \ ,&z'=\vert Z_{[k']}(\omega)\vert, \ k'>1,
		\end{cases}
	\end{align*}
	since $\hat{Z}_{1,n}=\arg\max_{\theta}I_n(\theta;\omega)$ by definition of the estimator. This is a contradiction for almost all $\omega\in\Omega$ as $R_{[1]}>R_{[k']}>0$ a.s. for all $k'>1$. In case that $\hat{Z}_{1,n}(\omega)$ diverges there exists a convergent subsequence $(\hat{Z}_{1,n_l}(\omega))_{l\in\mathbb{N}}$ that converges to an $z'$, which falls again into one of the two above cases, and the same contradiction arises. 
	Hence, we can conclude that $\hat{Z}_{1,n}\rightarrow \vert Z_{[1]}\vert$ a.s. as $n\rightarrow\infty$.
	
	Recall that $I_n(\vert Z_{[1]}(\omega);\omega\vert)=R_{[1]}^2(\omega)/4+O(n^{-2})$. Following the proof of \cite[Theorem 1]{hannan2}, it holds that 
	\begin{align*}
		0&\leq I_n(\hat{Z}_{1,n}(\omega);\omega) - I_n(\vert Z_{[1]}(\omega)\vert;\omega)\\
		&=\frac{R_{[1]}^2(\omega)}{4}\Bigg(\underbrace{\frac{\sin^2(\frac{1}{2}n(\vert Z_{[1]}(\omega)\vert-\hat{Z}_{1,n}(\omega)))}{n^2\sin^2(\frac{1}{2}(\vert Z_{[1]}(\omega)\vert-\hat{Z}_{1,n}(\omega)))}}_{\leq 1}-1\Bigg)+O\left(n^{-2}\right).
	\end{align*}
	Since the first term in the above is non-positive and the second term vanishes as $n\rightarrow\infty$, it follows that 
	$S_n^2(z_N)=\sin^2(nz_n)/(n^2\sin^2(z_n))\rightarrow 1$ as $n\rightarrow\infty$, where $z_n=(\vert Z_{[1]}(\omega)\vert-\hat{Z}_{1,n}(\omega))/2$. 
	The fact that $\sin^2(nz_n)$ is bounded by 1 and the convergence of the function $S_n^2(z_n)$ to $1$ implies that $n^2\sin^2(z_n)$ is also bounded by $1$ for $n$ large enough. 
	It follows that $\sin^2(z_n)\leq n^{-2}$ and hence $\vert z_n\vert\leq n^{-1}$ for $n$ large enough by the convergence $z_n\rightarrow0$. Equivalently, $n\vert z_n\vert \leq 1$ for large $n$, hence $\vert\sin(nz_n)/(nz_n)\vert$ can only converge to $1$ if $nz_n$ converges to $0$ as $n\rightarrow\infty$.  
	Ultimately, we have $n(\hat{Z}_{1,n}-\vert Z_{[1]}\vert)\rightarrow 0$ a.s.

	Next, the almost sure convergence of $\hat{a}_{1,n}$ and $\hat{b}_{1,n}$ to $a_{[1]}$ and $\hat{b}_{[1]}$ is established. Details can be found in \ref{appendixB} and \cite[Section 3]{walker}. 
	The explicit forms of $\hat{a}_{1,n}$ and $\hat{b}_{1,n}$ are given by 
	\begin{align*}
		\hat{a}_{1,n}=\frac{2}{n}\sum_{t=1}^nx(t)\cos(\hat{Z}_{1,n}t),\quad\hat{b}_{1,n}=\frac{2}{n}\sum_{t=1}^nx(t)\sin(\hat{Z}_{1,n}t)
	\end{align*}
	\cite[Equation (2.5)]{walker}. One can show that 
	\begin{align*}
		\left\vert\left(\hat{a}_{1,n}-a_{[1]}\right)+i\left(\hat{b}_{1,n}-b_{[1]}\right)\right\vert
		\leq \frac{2}{n}\left\vert M_n(\hat{Z}_{1,n}-\vert Z_{[1]}\vert)-n\right\vert + O\left(n^{-1}\right),
	\end{align*}
	where $\left\vert M_n(x)\right\vert =\sum_{t=1}^ne^{ixt}=e^{ix}\left(e^{inx}-1\right)/\left(e^{ix}-1\right)$ and $O(n^{-1})$ is a quantity that converges a.s. to 0 as $n\rightarrow\infty$. 
	Applying the mean value theorem and the fact that $\left\vert M_n'(x)\right\vert=\left\vert\sum_{t=1}^nte^{ixt}\right\vert < n^2$ for all $x\in\mathbb{R}$ \cite[p. 27]{walker} yields
	\begin{align*}
		n^{-1}\left\vert M_n(\hat{Z}_{1,n}-\vert Z_{[1]}\vert)-n\right\vert = n^{-1}\left\vert M_n(\hat{Z}_{1,n}-\vert Z_{[1]}\vert )-M_n(0)\right\vert < n\left\vert \hat{Z}_{1,n}-\vert Z_{[1]}\vert\right\vert \rightarrow 0 \quad a.s.
	\end{align*}
	as $n\rightarrow\infty$, which establishes the desired strong consistency of $\hat{a}_{1,n}$ and $\hat{b}_{1,n}$. 
	
	In the next step, compute the periodogram $I_n^{(1)}$ of $x^{(1)}(t)=x(t)-\hat{a}_{1,n}\cos(\hat{Z}_{1,n}t)-\hat{b}_{1,n}\sin(\hat{Z}_{1,n}t)$ and estimate $\hat{Z}_{2,n}$ from the location of the largest peak of $I_n^{(1)}$. As before, one then shows that $n(\hat{Z}_{2,n}-\vert Z_{[2]}\vert)\rightarrow 0$, $\hat{a}_{2,n}\rightarrow a_{[2]}$ and $\hat{b}_{2,n}\rightarrow b_{[2]}$ a.s. as $\rightarrow\infty$ and proceeds by setting $x^{(k)}=x^{(k-1)}(t)-\hat{a}_{k,n}\cos(\hat{Z}_{k,n}t)-\hat{b}_{k,n}\sin(\hat{Z}_{k,n}t)$ to compute $I_n^{(k)}$, establish the strong consistency of $\hat{Z}_{k,n}$ (with rate $n^{-1}$), $\hat{a}_{k,n}$ and $\hat{b}_{k,n}$ for $k=2$ and repeat this process for $k=3,\dots,N-1$. For more details we refer to \cite{hannan, walker} and \ref{appendixB}. 	
\end{proof}

\subsection{Kernel density estimator and weak consistency}
We denote by $\{\hat{Z}_{k,n}\}_{k=1}^N$ the estimators of the absolute frequencies $\{\vert Z_{[k]}\vert\}_{k=1}^N$ corresponding to the $N\in\mathbb{N}$ largest peaks located of the periodogram. 
Recall that frequencies $Z_k$, hence also the $Z_{[k]}$, are drawn independently from a distribution with the sought-after symmetric spectral density $f$ as probability density function. 

\begin{rem}
	\label{remabsZ}
	Let $Z$ be a random variable with symmetric probability density function $f$. Note that the probability density function of $\vert Z\vert$ satisfies the relation $f_{\vert Z\vert}(x)=2f(x)$ for any $x\geq 0$. 
\end{rem}
Hence, applying a kernel density estimator on the estimates $\{\hat{Z}_{k,n}\}_{k=1}^N$ yields an estimate of $2f$. Denote by $\hat{f}_N$ the kernel density estimator with kernel function $\kappa$ and bandwidth $h_N>0$:
\begin{align}
	\label{kdeZ}
	\hat{f}_N(x)=\frac{1}{Nh_N}\sum\limits_{k=1}^N\kappa\left(\frac{x-\vert Z_k\vert}{h_N}\right),
\end{align}
and define the kernel density estimator $\hat f_{N,n}$ based on $\left\{\hat Z_{k,n}\right\}_{k=1}^N$ by 
\begin{align}
	\label{kdeZhat}
	\hat f_{N,n}(x)=\frac{1}{Nh_N}\sum\limits_{k=1}^N\kappa\left(\frac{x-\hat Z_{k,n}}{h_N}\right).
\end{align}

Let $\Vert u\Vert_\infty=\sup_{x\in\mathbb{R}}\vert u(x)\vert$ denote the uniform norm of a function $u$ on $\mathbb{R}$. Furthermore, let $\ \stackrel{P}{\longrightarrow}$ denote convergence in probability and $\ \stackrel{a.s.}{\longrightarrow}$ almost sure convergence.

\begin{thm}
	\label{strongconsistencyfNnforfN}
	Let $X=\left\{X(t):t\in\mathbb{R}\right\}$ be a SRH $S\alpha S$ process with symmetric spectral density $f$,
	and $\left(x(1),\dots,x(n)\right)=\left(X(t_1),\dots,X(t_n)\right)$ be a sample of the process $X$ at equidistant points $t_1<\dots<t_n$, $n\in\mathbb{N}$.
	Furthermore, let $\{\hat{Z}_{k,n}\}_{k=1}^N$, $N\in\mathbb{N}$, be the strongly consistent frequency estimators of $\{\vert Z_{[k]}\vert\}_{k=1}^N$ from Theorem \ref{strongconsistencyZk}.
	Consider the kernel density estimator $\hat{f}_{N,n}$ defined in Equation (\ref{kdeZhat}) with Lipschitz continuous kernel function $\kappa$ and bandwidth $h_N>0$.
	Then, 
	\begin{align}
		\lim_{N\rightarrow\infty}\lim_{n\rightarrow\infty}h_N^2n\left\Vert\hat{f}_{N,n}-\hat{f}_N\right\Vert_\infty=0\quad a.s.
	\end{align}
\end{thm}
\begin{proof}
	Note, that the triangle inequality as well as the Lipschitz continuity of the kernel function $\kappa$ yield
	\begin{align*}
		\left\vert\hat f_{N,n}(x)-\hat{f}_N(x)\right\vert
		&\leq\frac{1}{Nh_N}\sum\limits_{k=1}^N\underbrace{\left\vert\kappa\left(\frac{x-\hat Z_{k,n}}{h_N}\right)-\kappa\left(\frac{x-\vert Z_{[k]}\vert}{h_N}\right)\right\vert}_{\leq L\left\vert\frac{x-\hat Z_{k,n}}{h_N}-\frac{x-\vert Z_{[k]}\vert}{h_N}\right\vert=\frac{L}{h_N}\left\vert \vert Z_{[k]}\vert-\hat Z_{k,n}\right\vert}\nonumber\\
		&\leq\frac{L}{Nh_N^2}\sum\limits_{k=1}^N\left\vert \vert Z_{[k]}\vert-\hat Z_{k,n}\right\vert,
	\end{align*}
	for all $x\in\mathbb{R}$. In particular, the above upper bound is uniform. 
	It  follows that 
	\begin{align}
		\label{strongconspart1}
		&\mathbb{P}\left(\lim_{N\rightarrow\infty}\lim_{n\rightarrow\infty}h_N^2n\left\Vert\hat{f}_{N,n}-\hat{f}_N\right\Vert_\infty>0\right)
		\leq\mathbb{P}\left(\lim_{N\rightarrow\infty}\lim_{n\rightarrow\infty}h_N^2n\frac{L}{Nh_N^2}\sum\limits_{k=1}^N\left\vert \vert Z_{[k]}\vert-\hat Z_{k,n}\right\vert>0\right)\nonumber\\
		=&\mathbb{P}\left(\lim_{N\rightarrow\infty}\lim_{n\rightarrow\infty}\frac{1}{N}\sum\limits_{k=1}^Nn\left\vert \vert Z_{[k]}\vert-\hat Z_{k,n}\right\vert>0\right)
	\end{align}
	The sets 
	$
	A_N=\left\{\lim_{n\rightarrow\infty}\frac{1}{N}\sum\limits_{k=1}^Nn\left\vert \vert Z_{[k]}\vert-\hat Z_{k,n}\right\vert>0\right\}
	$
	are nested with $A_N\subset A_{N+1}$, hence term (\ref{strongconspart1}) simplifies to 
	\begin{align*}
		&\lim_{N\rightarrow\infty}\mathbb{P}\left(\lim_{n\rightarrow\infty}\frac{1}{N}\sum\limits_{k=1}^Nn\left\vert \vert Z_{[k]}\vert-\hat Z_{k,n}\right\vert>0\right)
		=\lim_{N\rightarrow\infty}\mathbb{P}\left(\sum\limits_{k=1}^N\lim_{n\rightarrow\infty}n\left\vert \vert Z_{[k]}\vert-\hat Z_{k,n}\right\vert>0\right)\\
		\leq&\lim_{N\rightarrow\infty}\sum_{k=1}^N\underbrace{\mathbb{P}\left(\lim_{n\rightarrow\infty}n\left\vert \vert Z_{[k]}\vert-\hat Z_{k,n}\right\vert>0\right)}_{=0}=0.
	\end{align*}
\end{proof}	
Kernel density estimators have been studied extensively, and the following results on the consistency of $\hat{f}_N$ as well as the behavior of its bias and variance are well known, see \cite{silverman, kdeconsistency}.
\begin{lem}
	\label{lemmaconsistency}
	Let $h_N\rightarrow0$ with $Nh_N\rightarrow\infty$ as $N\rightarrow\infty$. Then, the kernel density estimator $\hat{f}_N$ is weakly pointwise consistent for $2f$ at every point of continuity $x$ of $f$. It is weakly uniformly consistent if $Nh_N^2\rightarrow\infty$ as $N\rightarrow\infty$. 
	
	If $\kappa$ is a right-continuous kernel function with bounded variation that vanishes at $\pm\infty$, $f$ is uniformly continuous, and the bandwidth satisfies $\sum_{N=1}^\infty\exp(-\gamma Nh_N^2)<\infty$ for any $\gamma>0$, then $\hat{f}_N$ is strongly uniform consistent.
\end{lem}
\begin{lem}
	\label{lemmabiasvar}
	Let $f$ be $r$-times continuously differentiable in a neighborhood of $x$. Let $\lim_{N\rightarrow\infty} Nh_N=\infty$. Then, $Bias(\hat{f}_N)=\mathbb{E}\left[\hat{f}_N(x)-2f(x)\right]=O(h_N^r)$ and $Var(\hat{f}_N(x))=O\left(\frac{1}{Nh_N}\right)$.
\end{lem}

As a direct consequence of Theorem \ref{strongconsistencyfNnforfN} and Lemma \ref{lemmaconsistency} the following consistency results of $\hat{f}_{N,n}$ for $2f$ can be given.  
\begin{cor}
	\label{kdecons}
	Consider the kernel density estimator $\hat{f}_{N,n}$ from Theorem \ref{strongconsistencyfNnforfN}. 
	Then, 
	for all $\varepsilon>0$ it holds that $$\lim_{N\rightarrow\infty}\lim_{n\rightarrow\infty}\mathbb{P}\left(\left\Vert\hat{f}_{N,n}-2f\right\Vert_{\infty}>\varepsilon\right)=0$$
	if $Nh_N^2\rightarrow\infty$ as $N\rightarrow\infty$.
	Under the weaker condition $Nh_N\rightarrow\infty$ as $N\rightarrow\infty$, weak pointwise consistency holds at any continuity point of $f$.
	
	Moreover,  under the conditions for strong uniform consistency of $\hat{f}_N$ in Lemma \ref{lemmaconsistency} it holds that 
	$$\mathbb{P}\left(\lim_{N\rightarrow\infty}\lim_{n\rightarrow\infty}\left\Vert\hat{f}_{N,n}-2f\right\Vert_\infty=0\right)=1.$$
\end{cor}
\begin{proof}
	The triangle inequality yields 
	\begin{align}
		\label{triangleinequalityconsistency}
		\left\Vert\hat{f}_{N,n}-2f\right\Vert_{\infty}\leq\left\Vert\hat{f}_{N,n}-\hat{f}_N\right\Vert_\infty+\left\Vert\hat{f}_N-2f\right\Vert_\infty,
	\end{align} 
	i.e. consistency of $\hat{f}_{N,n}$ for $2f$ is determined by its consistency for $\hat{f}_N$ as well as the consistency of $\hat{f}_N$ for $2f$. By Theorem \ref{strongconsistencyfNnforfN} strong and weak consistency, both pointwise as well uniformly, of $\hat{f}_{N,n}$ for $\hat{f}_N$ are guaranteed. Lemma \ref{lemmaconsistency} provides the conditions for weak pointwise and uniform consistency $\hat{f}_N$ for $2f$, as well as the conditions for strong uniform consistency.
\end{proof}

\begin{rem}
	\label{remconvrates}
	The asymptotic results in Lemma \ref{lemmabiasvar} can be used to derive rates for the pointwise weak consistency of $\hat{f}_N$ for $2f$. 
	The triangle inequality yields 
	\begin{align*}
		\left\vert\hat{f}_N(x)-2f(x)\right\vert\leq\left\vert\hat{f}_N(x)-\mathbb{E}\hat{f}_N(x)\right\vert+\left\vert \mathbb{E}\hat{f}_N(x)-2f(x)\right\vert,
	\end{align*}
	where the second summand on the right-hand side is the bias of $\hat{f}_N$ for $2f$. Applying Chebyshev's inequality to the first summand in the sum above yields 
	\begin{align}
		\label{kdepointwiseweak}
		\hat{f}_N(x)-2f(x)=O_p\left(\sqrt{\frac{1}{Nh_N}}+h_N^r\right)
	\end{align} 
	if $f$ is $r$-times continuously differentiable in a neighborhood of $x$, and the bandwidth $h_N$ satisfies $\lim_{N\rightarrow\infty}Nh_N=\infty$. 
	
	Under stronger assumptions, i.e. if $f$ is bounded, and the bandwidth $h_N\rightarrow0$ satisfies that $h_N/h_{2N}$ is bounded, $-log(h_N)/\log(\log(N))\rightarrow\infty$ and $Nh_N/\log(N)\rightarrow\infty$, then 
	\begin{align*}
		\left\Vert\hat{f}_N-\mathbb{E}\hat{f}_N\right\Vert_\infty=O\left(\sqrt{\frac{-\log(h_N)}{Nh_N}}\right)\quad a.s.,
	\end{align*}
	see \cite{einmahl,gineguillou}. If additionally $f$ is $r$-times continuously differentiable with bounded derivatives $f^{(i)}$, $i=1,\dots,r$, then 
	\begin{align}
		\label{kdeuniformstrong}
		\left\Vert\hat{f}_N-2f(x)\right\Vert_\infty=O\left(\sqrt{\frac{-\log(h_N)}{Nh_N}}+h_N^r\right)\quad a.s.
	\end{align}
\end{rem}

\begin{rem}
	\label{convratesfinal}
	Theorem 4 and Remark \ref{remconvrates} yield the following convergence rates.
	By Theorem \ref{strongconsistencyfNnforfN} it holds that $\Vert\hat{f}_{N,n}-\hat{f}_N\Vert_\infty=o(h_N^{-2}n^{-1})$ almost surely. In particular, this implies
	$\vert \hat{f}_{N,n}(x)-\hat{f}_N(x)\vert=o_p(h_N^{-2}n^{-1})$ 
	and Equation (\ref{kdepointwiseweak}) in Remark \ref{remconvrates} yields
	\begin{align*}
		\left\vert\hat{f}_{N,n}(x)-2f(x)\right\vert\leq\left\vert\hat{f}_{N,n}(x)-\hat{f}_N(x)\right\vert+\left\vert\hat{f}_N(x)-2f(x)\right\vert=o_p\left(h_N^{-2}n^{-1}\right)+O_p\left(\sqrt{\frac{1}{Nh_N}}+h_N^r\right)
	\end{align*}
	for the pointwise weak convergence rate of $\hat{f}_{N,n}(x)$ to $2f(x)$. 
	
	For the uniform strong consistency it holds that
	\begin{align*}
		\left\Vert\hat{f}_{N,n}-2f\right\Vert_{\infty}\leq\left\Vert\hat{f}_{N,n}-\hat{f}_N\right\Vert_\infty+\left\Vert\hat{f}_N-2f\right\Vert_\infty=o\left(h_N^{-2}n^{-1}\right)+ O\left(\sqrt{\frac{-\log(h_N)}{Nh_N}}+h_N^r\right)\quad a.s.
	\end{align*}
	by Equation (\ref{kdeuniformstrong}).
	Bandwidth choice heavily influences the performance of the kernel density estimation. It can be shown that the globally optimal bandwidth that minimizes the mean square error of the kernel density estimator behaves like $O(N^{-1/5})$. More details are given in Section \ref{numericsperiodogram}, in particular Equation (\ref{hopt}). The above convergence rates therefore simplify to 
	\begin{align*}
		\left\vert\hat{f}_{N,n}(x)-2f(x)\right\vert=o_p\left(\frac{N^{2/5}}{n}\right)+O_p\left(N^{-2/5}+N^{-r/5}\right)
	\end{align*}
	and 
	\begin{align*}
		\left\Vert\hat{f}_{N,n}-2f\right\Vert_{\infty}=o\left(\frac{N^{2/5}}{n}\right)+O\left(\sqrt{\frac{\log(N)}{5}}N^{-2/5}+N^{-r/5}\right)\quad a.s.
	\end{align*}
\end{rem}

\section{Numerical results}
\label{numericalanalysis}

In this section, we aim to verify our theoretical results with numerical inference. We will consider the following four examples for the spectral density of a SRH $S\alpha S$ process $X$. 
\begin{ex}
	\label{numexamples}
	We consider the following symmetric probability density functions on $\mathbb{R}$ as examples. 
	\begin{minipage}[t]{.33\textwidth}
		\vspace{-1ex}
		\begin{enumerate}
			\item \label{f1}$f_1(x)=\frac{1}{\sqrt{2\pi}}e^{-\frac{x^2}{2}}$ ,
			\item $f_2(x)=\frac{1}{4}x^2e^{-\vert x\vert}$ ,
		\end{enumerate}
	\end{minipage}%
	\begin{minipage}[t]{.5\textwidth}
		\begin{enumerate}[(a)]
			\setcounter{enumi}{2}
			\item $f_3(x)=\frac{1}{x^2}\mathbbm{1}_{[1,\infty)}(\vert x\vert )$ ,
			\item $f_4(x)=\frac{1}{2}\mathbbm{1}_{[-1,1]}(x)$ .
		\end{enumerate}
	\end{minipage}%
\end{ex}
For the simulation of the SRH $S\alpha S$ process $X$ we make use of the series representation in Proposition \ref{series}, i.e.
\begin{align}
	\label{simulationseries}
	X_N(t)=\left(C_\alpha b_\alpha^{-1}\right)^{1/\alpha}\sum\limits_{k=1}^N\Gamma_{k}^{-1/\alpha}\left(G_k^{(1)}\cos(tZ_k)+G_k^{(2)}\sin(tZ_k)\right),
\end{align}
where $\Gamma_k$ are the arrival times of a unit rate Poisson point process, $G_k^{(i)}$, $i=1,2$, are i.i.d. $N(0,1)$ and $Z_k$ are i.i.d. with probability density function $f_i$, $i=1,\dots,4$ from the example above. 
The constants $C_\alpha$ and $b_{\alpha}$ are given in Proposition \ref{series}.
Note that in the above series representation of $X$ the summation is finite up to $N\in\mathbb{N}$, where we choose $N$ large enough such that $\Gamma_k^{-1/\alpha}$ are negligibly small for $k\geq N$. 

\subsection{Independent paths}
\label{indeppaths}

In the introductory example in Section \ref{multpaths} we considered $L\in\mathbb{N}$ independent path realizations of a SRH $S\alpha S$ process $X$ with symmetric spectral density $f$. 
These paths build the basis for the estimation of the $\alpha$-sine transform of $f$. 
Applying the inversion method of the $\alpha$-sine transform described in \cite{viet} allows us to reconstruct the spectral density. 

For the results in Figure \ref{fig:alphasine} we simulated $L=100$ paths of the process $X$ with index of stability $\alpha=1.5$ and spectral density $f_1$ of Example \ref{numexamples}. 
The paths were sampled at $n=101$ equidistant points $0=t_1<...<t_{n}=T$ on the interval $[0,T]=[0,10]$. 
Across all paths samples of lags $X^{(l)}(t_i)-X^{(l)}(0)\sim S1.5S(\sigma(t_i))$ for $i=1,\dots,n$, $l=1,\dots,L$ were generated. 
The regression-type estimators of Koutrouvelis \cite{koutrouvelis} were used to estimate $\alpha$ and the scale parameters $\sigma(t_i)=\left\Vert X(t_i)-X(0)\right\Vert_\alpha$ of the lags. 
With Equation (\ref{Tf_est}) we get estimates of the $\alpha$-sine transform of $f$ at the points $\{t_i/2\}_{i=1}^n$, and the inversion method in \cite{viet} yields an estimate of the spectral density $f$. 
One might incorporate smoothing for better estimation results, see \cite[Section 6.2.1]{viet}, see Figure \ref{fig:alphasine} (c). 
The performance clearly improves with the increase of the number of paths $L$, the number of sample points $n$ and the sample range $T$, but further analysis was omitted here, since our focus lies on statistics based on a single path of SRH $S\alpha S$ processes. 
\begin{figure}[ht]
	\centering
	\begin{subfigure}{0.33\textwidth}
		\includegraphics[width=\textwidth]{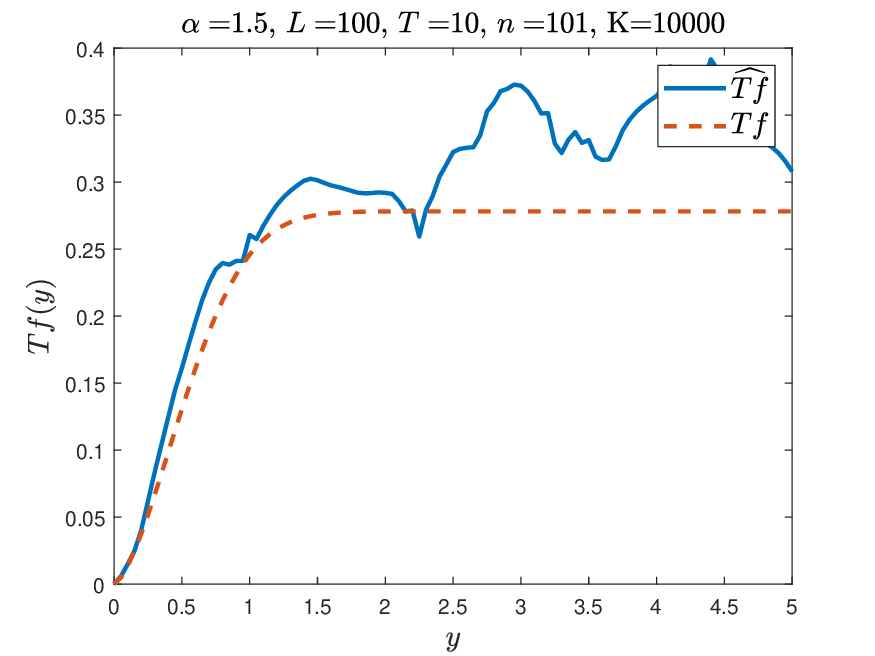}
		\caption{Estimate of the $\alpha$-sine transform \texorpdfstring{$T_\alpha f_1$}{Tf1}}
	\end{subfigure}
	\begin{subfigure}{0.33\textwidth}
		\includegraphics[width=\textwidth]{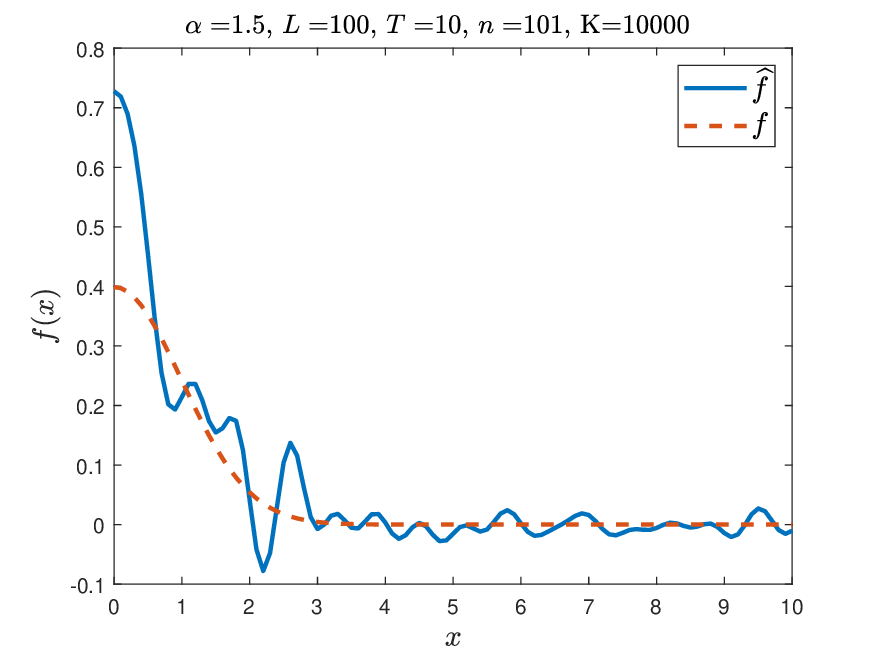}
		\caption{Estimate of the spectral density \texorpdfstring{$f_1$}{f1}.}
	\end{subfigure}
	\begin{subfigure}{0.33\textwidth}
		\includegraphics[width=\textwidth]{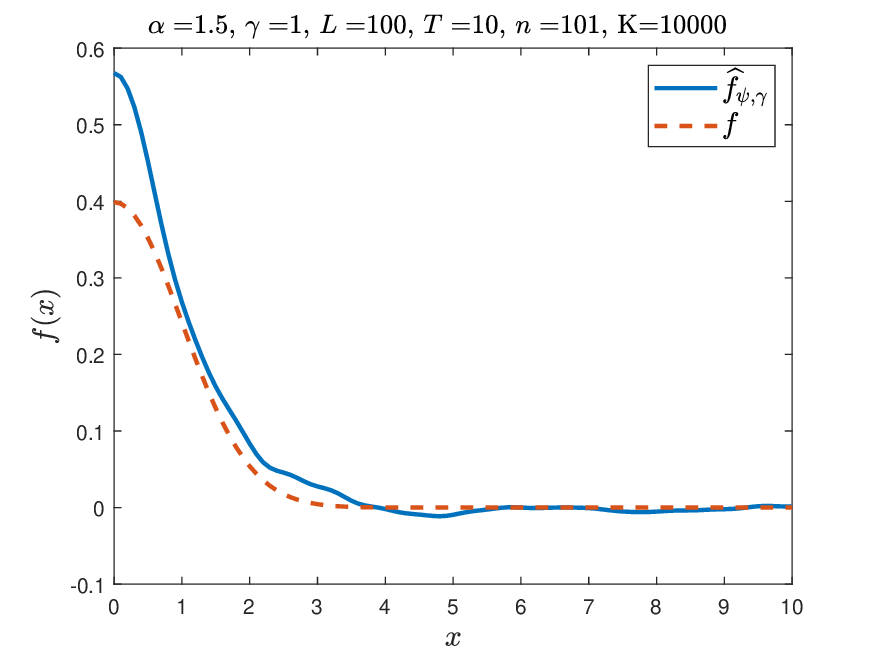}
		\caption{Smoothed estimate of the spectral density \texorpdfstring{$f_1$}{f1}.}
	\end{subfigure}
	\caption{\textbf{Inference on multiple paths}. The solid blue line shows the estimates and the dashed red line shows the theoretical function. 
	}
	\label{fig:alphasine}
\end{figure}

\subsection{Periodogram frequency estimation}
\label{numericsperiodogram}
The periodogram estimate can be jeopardized by errors caused by \emph{aliasing} and \emph{spectral leakage} \cite[Chapters 4 and 10]{signalprocessing}, when both the range $T$ on which the signal is sampled as well as number of sample points $n$ are small. 
Determining all peaks in the periodogram at once proves to be a difficult task. 
Distinguishing between peaks in the periodogram which actually stem from frequencies and not spectral leakage requires a meticulous setup of tuning parameters such as the minimum distance between peaks or their minimum height when employing algorithms such as \texttt{findpeaks} in \texttt{Matlab}.
Instead we utilize an iterative approach to estimate the peak locations from the periodogram as in Theorem \ref{strongconsistencyZk}, see also \cite[Chapter 3.2, p.53]{hannan}. 

The choice of $N$, i.e. the number of iterations, therefore the number of frequencies to be estimated, depends on the sample size $n$ as follows.
The  convergence rate of the kernel density estimator $\hat{f}_{N,n}$ depends on both the sample size $n$ and the number $N$ of estimated frequencies, see Remark \ref{convratesfinal}. In particular, the convergence of the ratio $N^{2/5}/n$ to $0$ is crucial. Since the sample size $n$ is fixed for a given path observation, we chose $N$ such that $N^{2/5}/n\approx\varepsilon$ for some small error $\varepsilon>0$.
The resulting estimate $\hat{Z}_{1,n},\dots,\hat{Z}_{N,n}$ will be used for the kernel density inference of the spectral density $f$.

We use \texttt{Matlab}'s \texttt{findpeaks} function for to detect the peaks and their locations in the periodogram. The specific value of the function's parameter \texttt{MinPeakProminence} is chosen such that the above iteration does not break before delivering $N$ frequency estimates.  

The computation of the periodogram is performed with \emph{zero-padding} \cite[Chapter 8]{signalprocessing}, i.e. zeros are added at the end of the sample resulting in a new sample $(x(1),\dots,x(n),0,\dots,0)$ of length $\tilde{n}>n$. Hence, the discrete Fourier transform is computed on a finer grid $\left\{2\pi j/\tilde{n}\right\}_{j=1}^{\tilde{n}}$ resulting in an interpolation of the periodogram between the actual Fourier frequencies $\left\{2\pi j/n\right\}_{j=1}^n$, which allows us to better distinguish between peaks. Zero-padding has no noticable effect on the computation time of the periodogram estimate. It is practical to choose $\tilde{n}$ to be a power of 2 to make use of the (Cooley-Tukey or radix-2) fast Fourier transform and its reduced complexity of $O(\tilde{n}\log\tilde{n})$ compared to the direct computation of the discrete Fourier transform \cite[Chapter 9]{signalprocessing}.
For our examples the samples are zero-padded with $\tilde{n}=2^{13}=8192$ if $n<\tilde{n}$.

We compute kernel density estimator $\hat{f}_{N,n}$ given in Equation (\ref{kdeZhat}) with the Gaussian kernel $\kappa(x)=e^{-x^2/2}/\sqrt{2\pi}$. 
As for the bandwidth, any fixed bandwidth $h=h_N$ immediately fulfills all the conditions for the consistency results in Lemma \ref{lemmaconsistency}. 
But kernel density estimation is highly sensitive to bandwidth selection, and a poor choice of $h$ naturally leads leads to poor performance of the estimator.
A global optimal bandwidth, which minimizes the mean integrated squared error of the kernel density estimator, i.e. the $\text{MISE}(\hat f_N)=\int_\mathbb{R}\text{MSE}(\hat f_N(x))dx$, where $\text{MSE}(\hat f_N(x))=\mathbb{E}[(\hat f_N(x)-2f(x))^2]$ is given by 
\begin{align}
	\label{hopt}
	h^\ast_N=N^{-1/5}\left(\int_\mathbb{R}\kappa(t)dt\right)^{1/5}\left(\int_{\mathbb{R}}t^2\kappa(t)dt\right)^{-2/5}\left(\int_{\mathbb{R}}2f''(x)dx\right)^{-1/5}=O\left(N^{-1/5}\right).
\end{align}
For the above optimal bandwidth the $\text{MISE}$ and the $\text{MSE}$ are of order $O(N^{-5/4})$, see \cite[Chapter 3.3]{silverman}.

The optimal bandwidth $h^\ast_N$ in Equation (\ref{hopt}) is nice for theoretical purposes but is not applicable in practice due to the simple fact that it depends on the unknown spectral density $f$. 
Instead, many methods to approximate the optimal bandwidth can be found in literature. \emph{Scott's} or \emph{Silverman's rules of thumb} are widespread in practices for their ease of use but assume the sample to be drawn from a Gaussian distribution \cite[Section 3.4]{silverman}. 
These methods are efficient and fast but their accuracy can only be guaranteed in the Gaussian case or for the estimation of unimodal and close to Gaussian densities.
When the form of the sought-after density is unknown, which is usually the case, methods like the \emph{unbiased cross-validation} \cite{bowman, rudemo} or the \emph{Sheather and Jones  plug-in method} \cite{sheather} are far better suited. 
There are many more cross-validation and plug-in methods at hand, see e.g. \cite{sheather2} for an overview, but we applied the Sheather and Jones method as it is already implemented in \texttt{R}. Similar results were achieved with other methods.

We consider one path of a SRH $S1.5S$ process for examples $f_1,\dots,f_4$ in Figure \ref{fig:allfa15}. The path is sampled on the interval $[0,T]$, $T=5000$, at $n=10^4$ equidistant points. 
The number of frequencies in the series representation of $X$ is set to $K=10^4$. 
In Remark \ref{convratesfinal} we derived convergence rates for the kernel density estimator $\hat{f}_{N,n}$. We set the number of estimated frequencies $N$ such that $N^{2/5}/n=\varepsilon$ for some small $\varepsilon>0$. 
Choosing $\varepsilon=10^{-3}$ yields $N= 316$. 

The results for the spectral density estimation are given in Figure \ref{fig:allfa15}.
Similar results were achieved using other well-known kernel functions such as the Epanechnikov kernel or triangle 
kernel. 
\begin{figure} 
	\centering
	\begin{subfigure}{0.245\textwidth}
		\includegraphics[width=\textwidth]{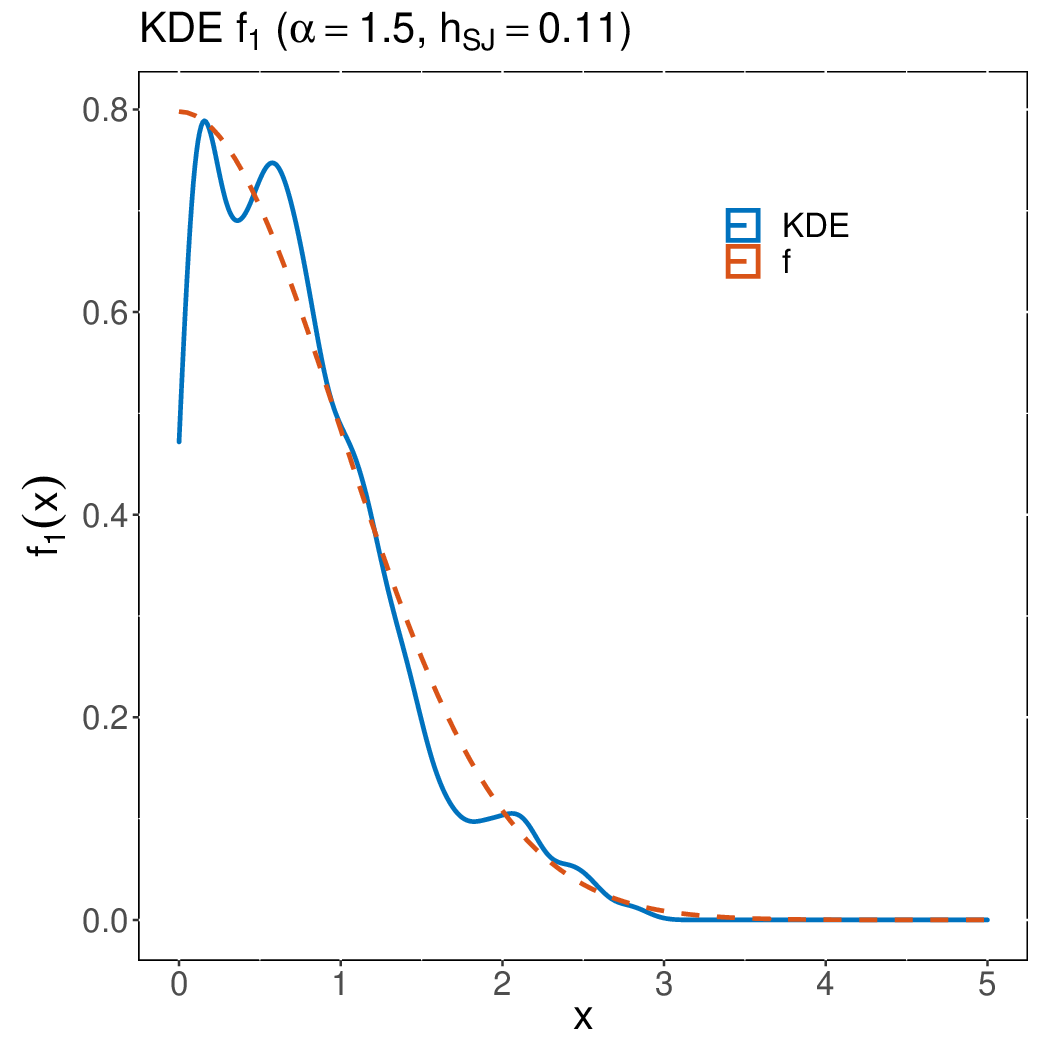}
		\caption{\texorpdfstring{$f_1$}{kdef1}}
	\end{subfigure}
	\begin{subfigure}{0.245\textwidth}
		\includegraphics[width=\textwidth]{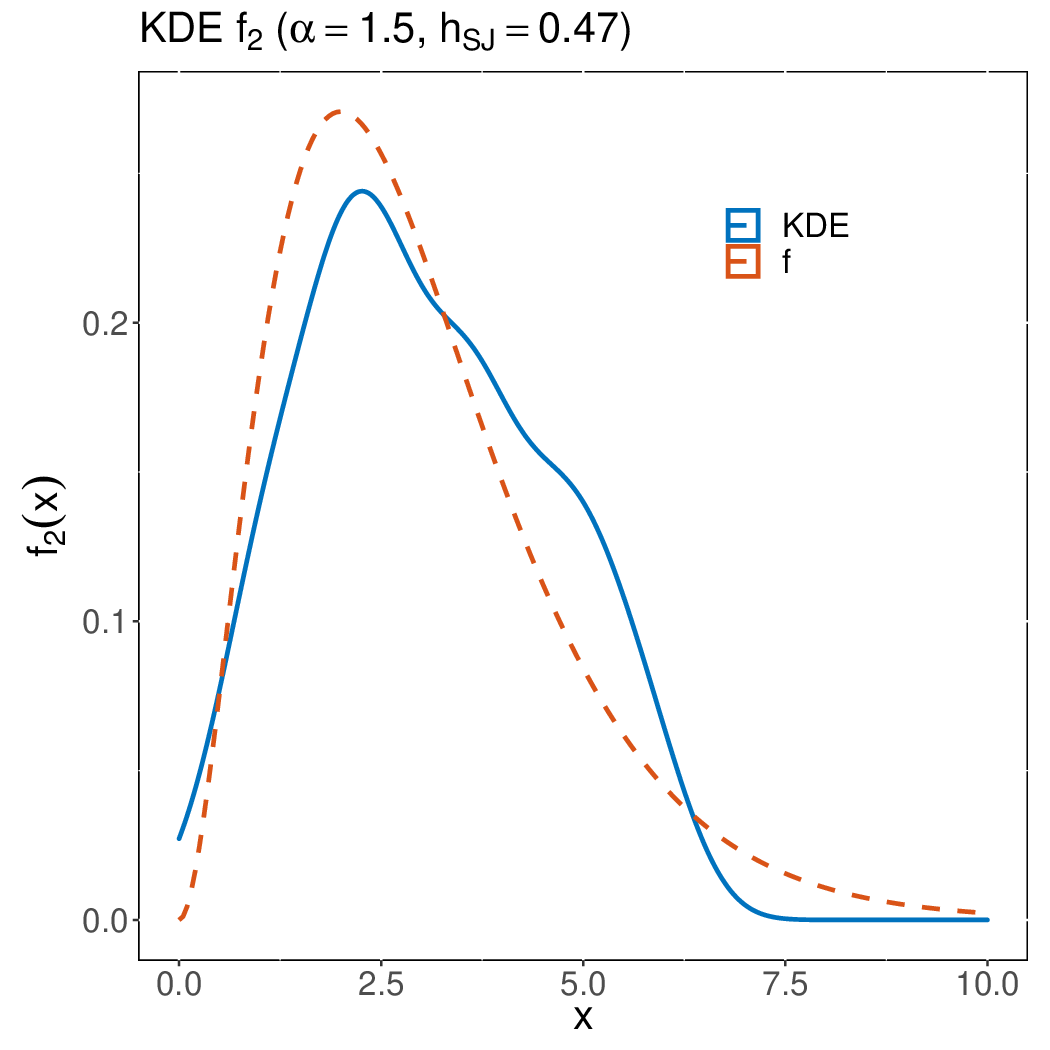}
		\caption{\texorpdfstring{$f_2$}{kdef2}}
	\end{subfigure}
	\begin{subfigure}{0.245\textwidth}
		\includegraphics[width=\textwidth]{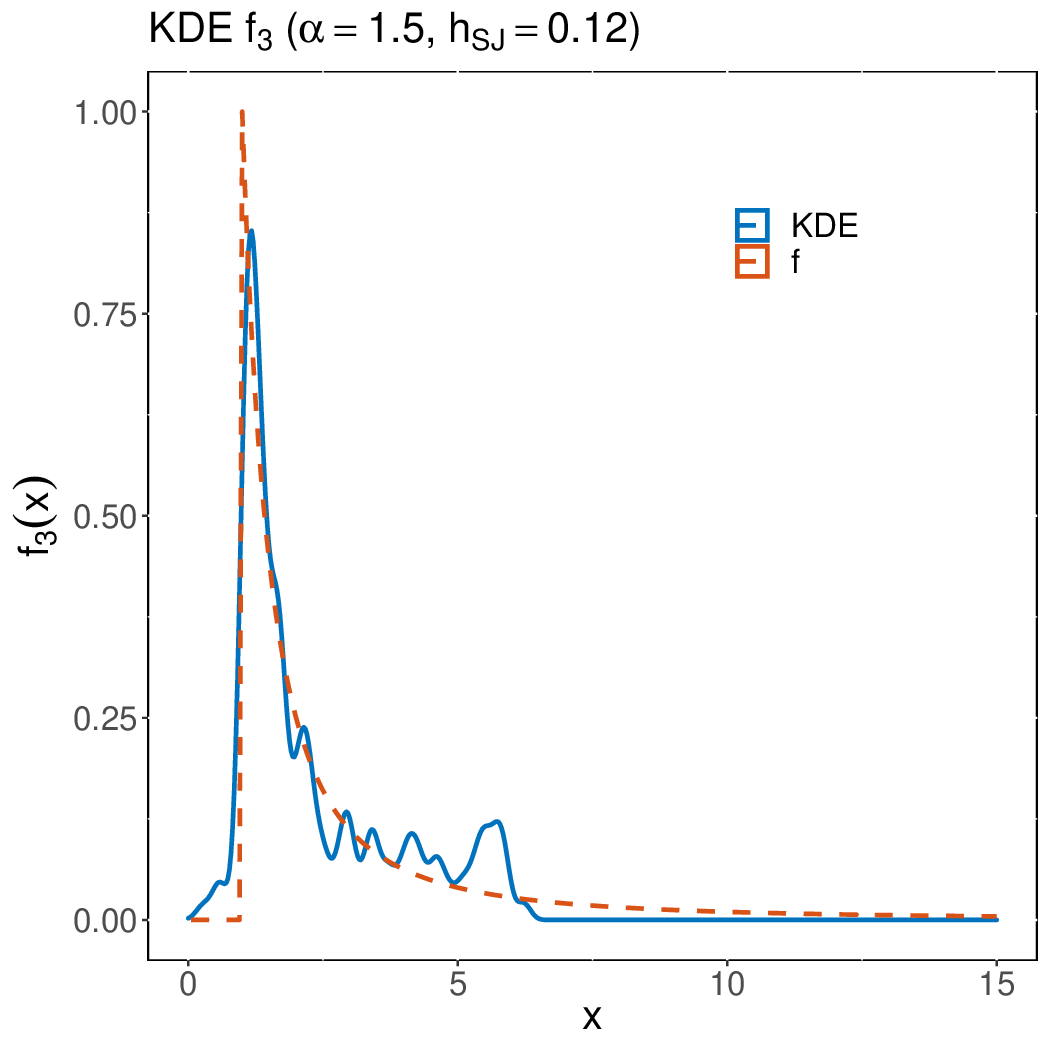}
		\caption{\texorpdfstring{$f_3$}{kdef3}}
	\end{subfigure}
	\begin{subfigure}{0.245\textwidth}
		\includegraphics[width=\textwidth]{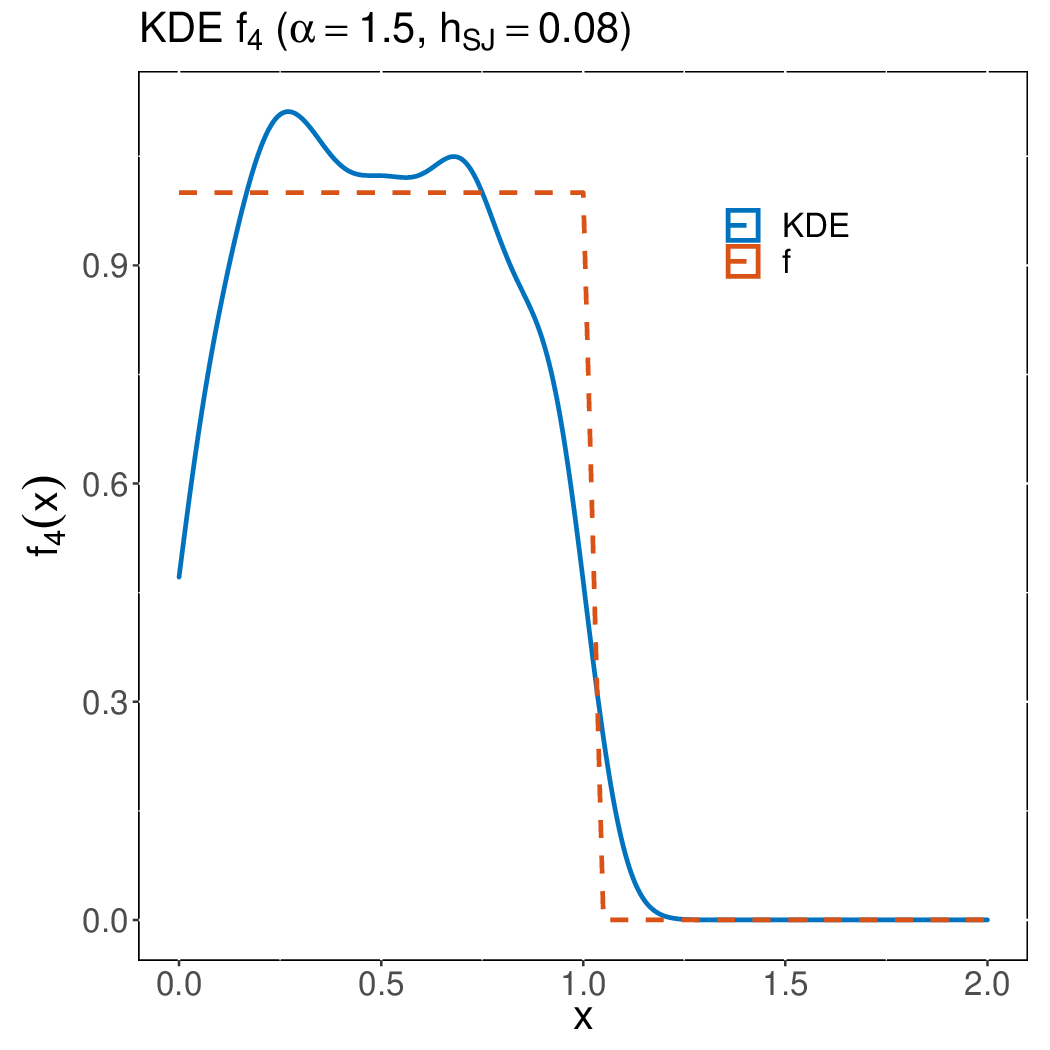}
		\caption{\texorpdfstring{$f_4$}{kdef4}}
	\end{subfigure}
	\caption{\textbf{Spectral density estimation} (Gaussian kernel, Sheather-Jones bandwidth). Kernel density estimator (blue) against spectral density (red) for examples $f_1,\dots,f_4$ with index of stability $\alpha=1.5$, sample range $T=5000$, sample size $n=10^4$ and number of estimated frequencies $N=316$.}
	\label{fig:allfa15}
\end{figure}
Figure \ref{fig:f2otheralphas} illustrates our results for the estimation of example function $f_2$ in the cases $\alpha\in\{1.75,1.25,0.75,0.25\}$. 
As expected, the estimation becomes more difficult as the index of stability $\alpha$ decreases. 
Recall that the decay of the amplitudes $R_k$ in the series representation of the SRH $S\alpha S$ process $X$ in Equation (\ref{amplphasefreq}) is determined by the factors $\Gamma_k^{-1/\alpha}$.
The smaller $\alpha$, the faster $\Gamma_k^{-1/\alpha}$ tends to $0$, which directly translates to the amplitudes $R_k$. 
Since the $R_k$ decay much faster it becomes increasingly difficult to estimate the corresponding frequencies $Z_k$ with smaller values of $\alpha$. 
In Figure \ref{fig:f2otheralphas} (d) we can see a dominating frequency, most likely associated to the largest amplitude $R_{[1]}\gg R_{[k]}$, $k\geq 2$. 
The periodogram will oscillate heavily in a neighbourhood of that frequency such that all other frequencies with amplitudes $R_k\ll R_{[1]}$ are hard to detect. 
Larger sample sizes $n$, which allow for larger $N$ while maintaining the same error bound $\varepsilon$, as well as the application of smoothing window functions lead to better results for small $\alpha$.
\begin{figure} 
	\centering
	\begin{subfigure}{0.245\textwidth}
		\includegraphics[width=\textwidth]{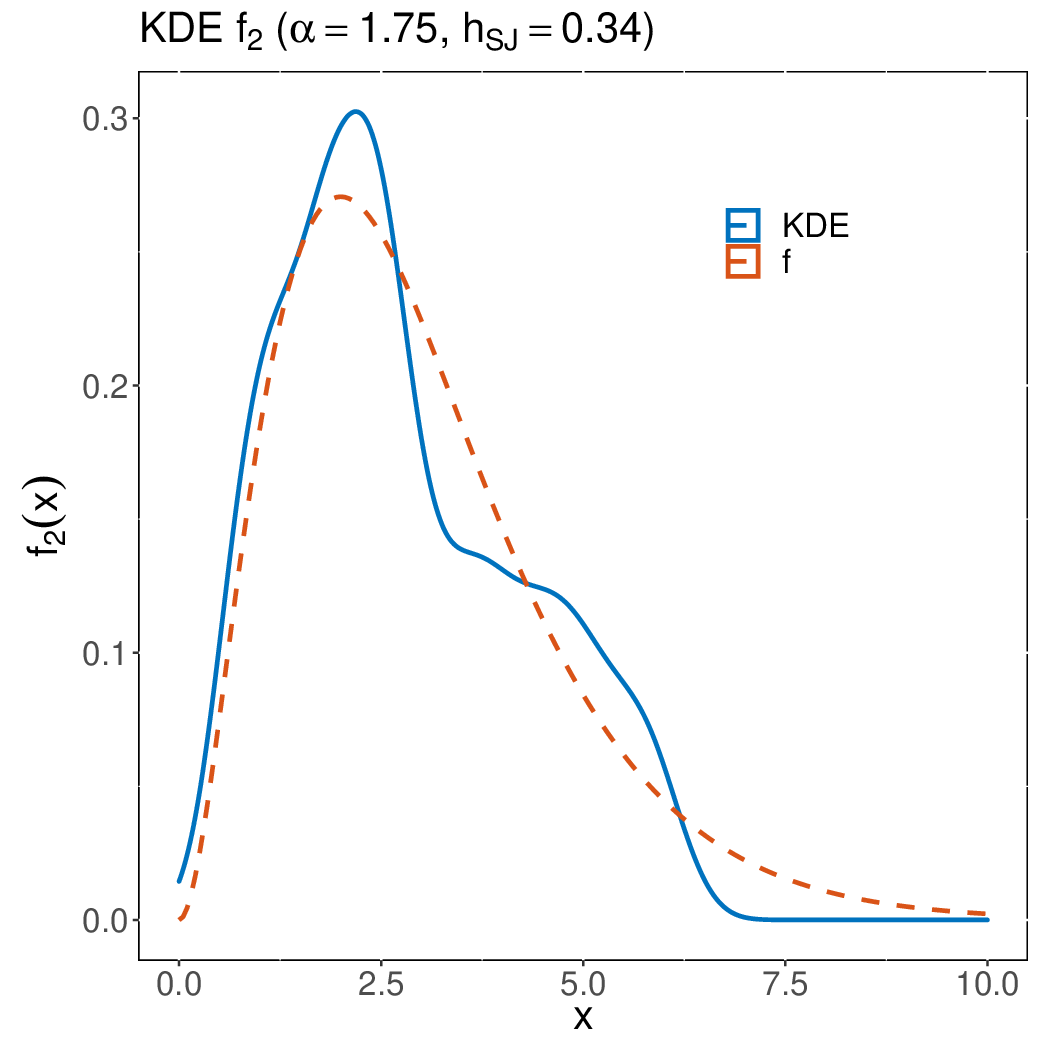}
		\caption{\texorpdfstring{$\alpha=1.75$}{kdealpha125}}
	\end{subfigure}
	\begin{subfigure}{0.245\textwidth}
		\includegraphics[width=\textwidth]{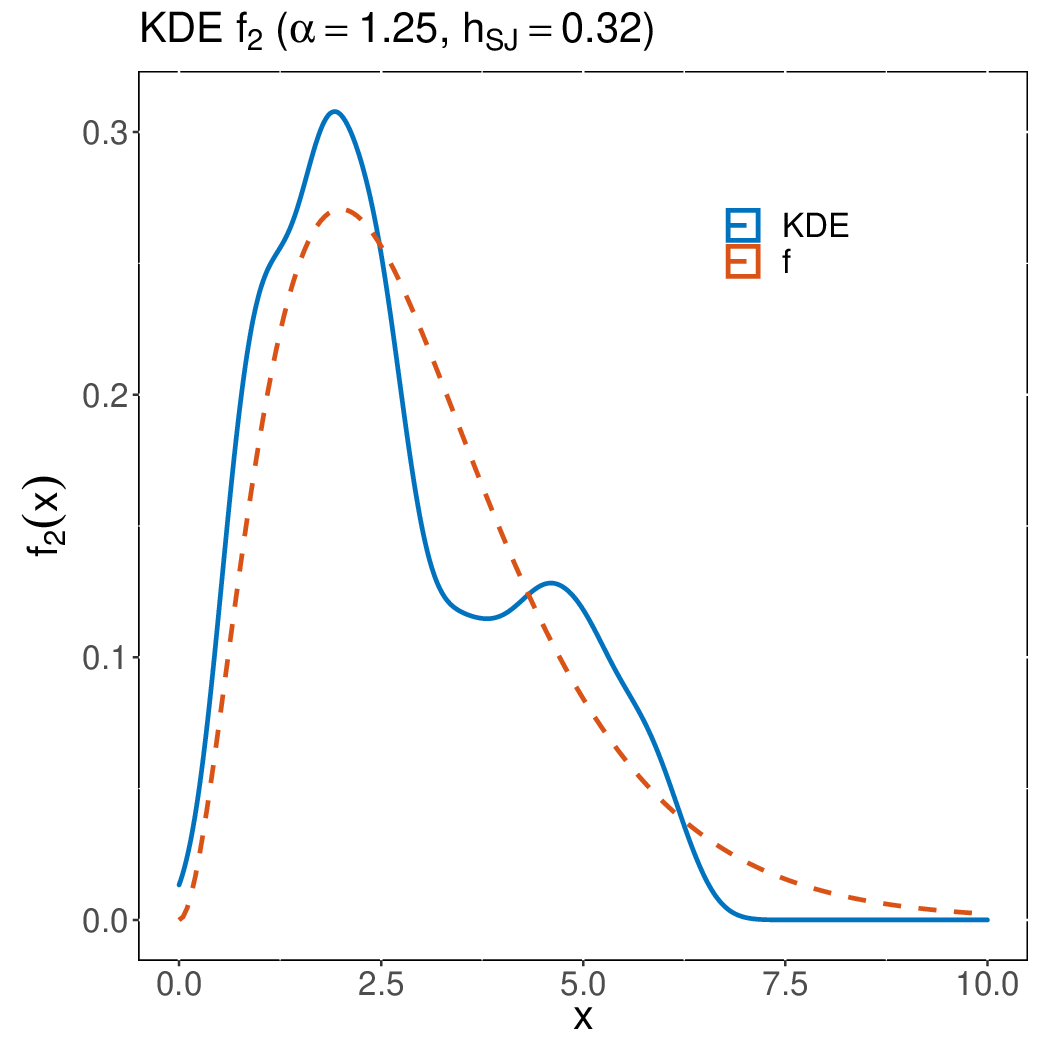}
		\caption{\texorpdfstring{$\alpha=1.25$}{kdealpha175}}
	\end{subfigure}
	\begin{subfigure}{0.245\textwidth}
		\includegraphics[width=\textwidth]{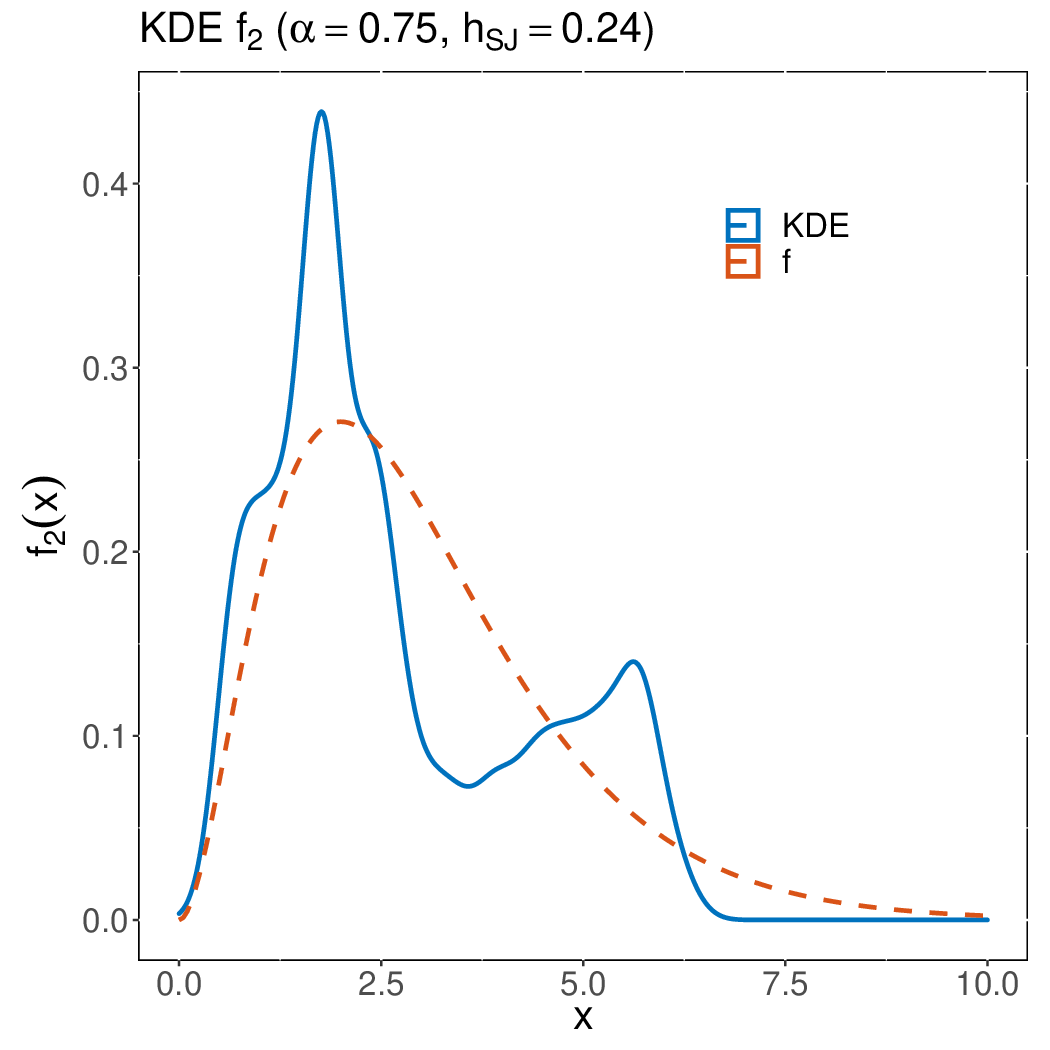}
		\caption{\texorpdfstring{$\alpha=0.75$}{kdealpha075}}
	\end{subfigure}
	\begin{subfigure}{0.245\textwidth}
		\includegraphics[width=\textwidth]{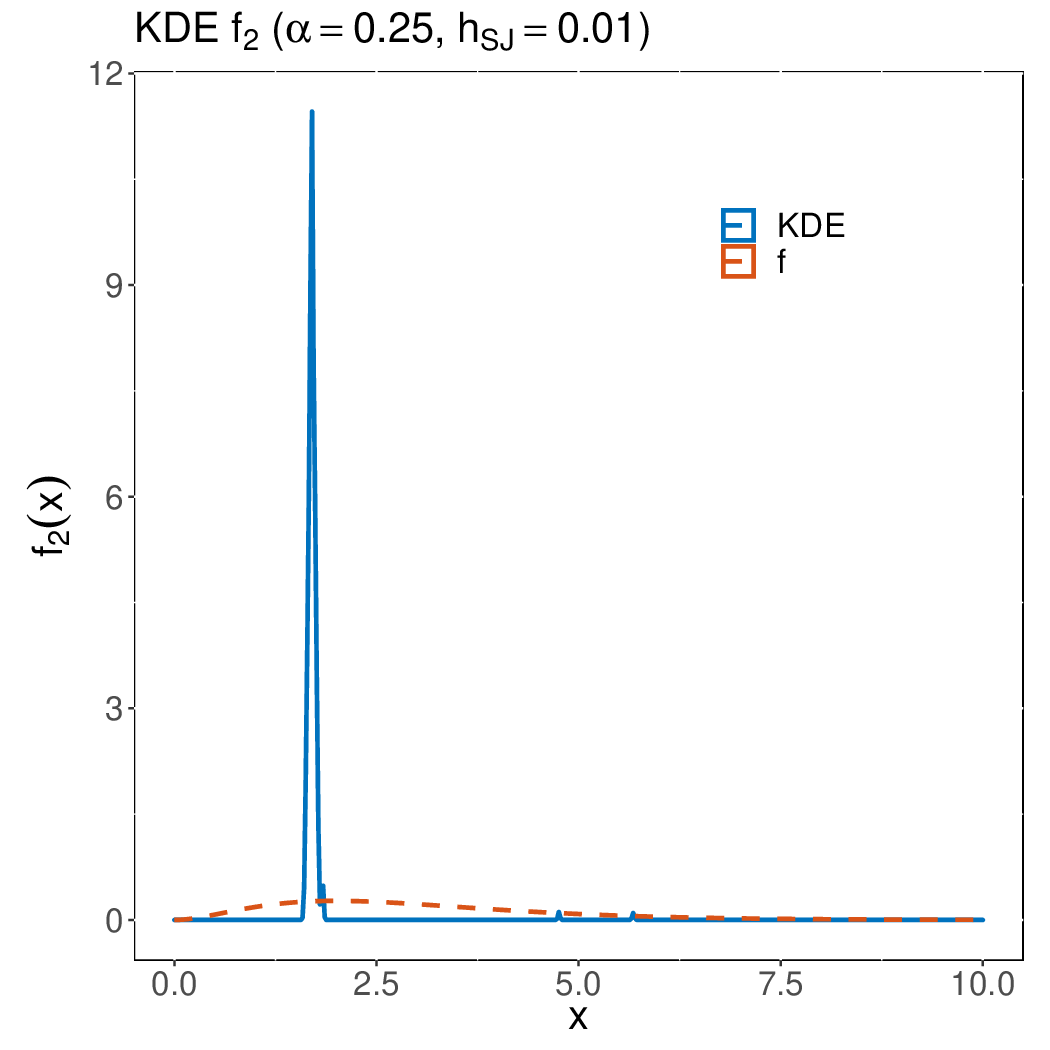}
		\caption{\texorpdfstring{$\alpha=0.25$}{kdealpha025}}
	\end{subfigure}
	\caption{\textbf{Spectral density estimation} (Gaussian kernel, Sheather-Jones bandwidth). 
		Kernel density estimator (blue) against spectral density (red) for example $f_2$ with index of stability $\alpha\in\{1.75,1.25,0.75,0.25\}$, sample range $T=5000$, sample size $n=10^4$ and number of estimated frequencies $N=316$.}
	\label{fig:f2otheralphas}
\end{figure}

One thing we would like to mention is the the dip of the kernel density estimates near $0$ in examples $f_1$ and $f_4$, see Figures \ref{fig:allfa15} (a) and (d). These dips are a direct consequence of a commonly known problem in frequency estimation, i.e. very low frequencies that are close to $0$ are difficult to detect via the periodogram. This can be resolved by an increase of the sampling range $T$  and sample size $n$. For more details we refer to \cite[Section 3.5]{hannan}.
\begin{table}
	\renewcommand{\arraystretch}{1.25}
	\setlength{\tabcolsep}{3pt}
	\begin{tabular}{|c|c||ccc|ccc|ccc|}
		\hline
		\multirow{2}{*}{Mean $L^2$-dist. $\times 10^{-1}$}&$n$&\multicolumn{3}{|c|}{$10^3$}&\multicolumn{3}{|c|}{$5\cdot10^3$}&\multicolumn{3}{|c|}{$10^4$}\\
		&$N$&10&20&50&50&75&100&100&200&300\\
		\hline
		\multirow{4}{*}{$\alpha=0.75$}&$f_1$&3.66 & 2.74 & 2.13&3.29 & 2.98 & 2.77&3.92 & 3.19 & 2.79\\
		&$f_2$&2.26 & 1.63 & 1.19&2.11 & 1.86 & 1.71&2.58 & 2.10 & 1.87\\
		&$f_3$&6.79 & 6.99 & 7.15&7.32 & 7.22 & 7.13&7.46 & 7.14 & 6.96\\
		&$f_4$&4.42 & 3.37 & 2.41&4.13 & 3.70 & 3.35&4.85 & 3.84 & 3.23\\
		\hline
		\multirow{4}{*}{$\alpha=1.5$}&$f_1$&2.49&1.85&1.67&1.65 & 1.48 & 1.38&1.49 & 1.25 & 1.17\\
		&$f_2$&1.54&1.07&0.69&0.85 & 0.71 & 0.63&0.72 & 0.55 & 0.48\\
		&$f_3$&6.31 &6.58& 6.98&6.68 & 6.56 & 6.46&6.54 & 6.27 & 6.23\\
		&$f_4$&2.85 &2.14 &1.77&2.08 & 1.87 & 1.75&1.92 & 1.61 & 1.45\\
		\hline
	\end{tabular}
	\caption{Mean $L^2$-distances between kernel density estimators and spectral densities $f_1,\dots,f_4$ based on  $L=1000$ single path simulations and spectral density estimation for each path, respectively. For example $f_1$ the $L^2$-distance is evaluated on the interval $[0,5]$, for $f_2$ on $[0,10]$, for $f_3$ on $[0,15]$ and for $f_4$ on $[0,2]$.}
	\label{tab:comparison}
\end{table}

Table \ref{tab:comparison} gives an overview of the mean $L^2$-distance between the kernel density estimator $\hat{f}_{N,n}$ and the true spectral density $f_i$.
For example $f_1$ the $L^2$-distance is evaluated on the interval $[0,5]$, for $f_2$ on $[0,10]$, for $f_3$ on $[0,15]$ and for $f_4$ on $[0,2]$.
For each combination of example $f_i$, index of stability $\alpha$, sample size $n$ and number of estimated frequencies $N$,  $L=1000$ single paths were simulated. From each of those paths a spectral density estimate and the corresponding $L^2$-distance to the true spectral density is computed.  
We clearly observe that the estimators perform better the larger the index of stability $\alpha$ is. 
This is to be expected as already mentioned before. 

Within a fixed sample size $n$ we also see that the $L^2$-distance between the estimators and the respective spectral density decreases with increasing $N$. 
Additionally, for $\alpha=1.5$ the results also improve with increasing sample size $n$. 

In the case of $\alpha=0.75$ we see a decrease in accuracy for examples $f_1$, $f_3$ and $f_4$ comparing the same number of estimated frequencies $N=100$ for both $n=5\cdot10^3$ and $n=10^4$. 
The reason is that for small $\alpha$ a small number frequencies with relatively large amplitudes dominate all other frequencies, due to the fast decay of $\Gamma_k^{-1/\alpha}$. 
In combination with larger sample sizes this leads to oscillations in the periodogram around these dominating frequencies which are not completely filtered out by the iteration process and overshadow frequencies with smaller amplitudes.

Other issues that can lead to poorer estimation results are discontinuities, see $f_3$ and $f_4$, as well as spectral densities which generate frequencies close to $0$ that are difficult to detect by the periodogram, see $f_1$ and $f_4$. 

Possible solutions to the aforementioned issues might be smoothing of the periodogram using appropriate window functions as well as adaptive bandwidth methods with narrower bandwidths at discontinuities and wider bandwidths where the kernel density estimator oscillates. 

Also note that, out of all examples, estimation results for $f_2$ are the most accurate due to it vanishing at the origin such that no frequencies close to $0$ need to be estimated, and its smoothness, which has a direct effect on the convergence rate of the kernel density estimator, see Remark \ref{convratesfinal}. 

\section{Conclusion}

Harmonizable $S\alpha S$ processes are one of the three main classes of stationary $S\alpha S$ processes. 
Unlike stationary moving average $S\alpha S$ processes, the harmonizable case has not received much attention from a statistical point of view. 
This is mainly due to the non-ergodicity of these processes, which inhibits the application of standard empirical methods. 

We considered the special case of stationary real harmonizable $S\alpha S$, in which the circular control measure of the process is the product of the uniform probability measure on the unit circle $(0,2\pi)$ and a finite control measure $m$.
Assuming that the control measure $m$ has a symmetric density $f$, the goal of our work was to develop a consistent and efficient statistical procedure to estimate $f$.

The series representation in Proposition \ref{series} shows that a SRH $S\alpha S$ process $X$ is a conditional non-ergodic harmonic Gaussian process. 
In Theorem \ref{nonerglim} and Remark \ref{lagprocess} we derived the non-ergodic limits of the empirical characteristic function of $X$ and the lag process $\left\{X(t+h)-X(t):t\in\mathbb{R}\right\}$, $h>0$. 
These limits can be explicitly given in terms of the Bessel function of the first kind of order $0$ and the processes' invariant sets.

Additionally, Theorem \ref{thmamplphasefreq} also yields an equivalent series representation of $X$ in terms of amplitudes $\left\{R_k\right\}_{k=1}^\infty$, i.i.d. uniform phases $\left\{\Theta_k\right\}_{k=1}^\infty$ and frequencies $\left\{Z_k\right\}_{k=1}^\infty$, see Equation (\ref{amplphasefreq}). 
The frequencies $Z_k$ are i.i.d. with probability density function $f$.
Although these quantities are random, they are predetermined for a given path. 

In signal theory, amplitudes, phases and frequencies are usually assumed to be deterministic, and randomness is added by some stationary ergodic noise sequence. In the case of SRH $S\alpha S$ processes, randomness is introduced by the random amplitudes, phases and frequencies themselves. Furthermore, path observations are non-ergodic. 
The methods we employ for frequency estimation are not new but have not been applied in this context before. We show that consistency is given since the frequencies to be estimated form an i.i.d. sample.

The periodogram is a fast and efficient tool for the estimation of the absolute frequencies $\vert Z_k\vert$ as it relies on the fast Fourier algorithm.
We show that the frequency estimators are strongly consistent in Theorem \ref{strongconsistencyZk}. Applying kernel density estimation yields an estimate of the spectral density. 
Under minimal assumptions on the kernel function, the kernel density's bandwidth and the spectral density $f$, various consistency results for our spectral density estimator are proven in Theorem \ref{strongconsistencyfNnforfN} and Corollary \ref{kdecons}. Convergence rates are given in Remark \ref{convratesfinal}.

An extensive numerical analysis shows that our proposed estimation method performs well on a variety of examples with index of stability $\alpha>0.5$. 
The smaller the index of stability $\alpha$ is, the more difficult the estimation of the spectral density becomes as peaks in the periodogram are harder to detect. 
In general, an increase of the sampling range $T$ and sample size $n$ results in significant improvements of the estimation.
Problems might arise when the spectral density has discontinuities or does not vanish at $0$, as seen in examples $f_1$, $f_3$ and $f_4$. 
Further improvements can be achieved with the application of different window functions in the periodogram computation as well as adaptive bandwidth methods for the kernel density estimation. 

Ultimately, aside from the consistency of our spectral density estimator and its ease of computation, we would like to highlight that no other requirements or prior knowledge on the process (such as e.g. the index of stability $\alpha$) is needed for the estimation of the spectral density $f$.  
The \texttt{Matlab} and \texttt{R} implementations used in this paper can be found in \cite{code}.


\appendix




 \bibliographystyle{model4-names}

\bibliography{bibliography}

\begin{thebibliography}{31}
\expandafter\ifx\csname natexlab\endcsname\relax\def\natexlab#1{#1}\fi
\providecommand{\bibinfo}[2]{#2}
\ifx\xfnm\undefined \def\xfnm[#1]{\unskip,\space#1}\fi
\bibitem[{{Basse-O\'{}Connor, Andreas} et~al.(2021){Basse-O\'{}Connor,
  Andreas}, {Gr\o{}nb\ae{}k, Thorbj\o{}rn} and {Podolskij, Mark}}]{basse}
\bibinfo{author}{{Basse-O\'{}Connor, Andreas}\xfnm[]},
  \bibinfo{author}{{Gr\o{}nb\ae{}k, Thorbj\o{}rn}\xfnm[]},
  \bibinfo{author}{{Podolskij, Mark}\xfnm[]}.
\newblock \bibinfo{title}{Local asymptotic self-similarity for heavy tailed
  harmonizable fractional l\'evy motions}.
\newblock \bibinfo{journal}{ESAIM: PS}
  \bibinfo{year}{2021};\bibinfo{volume}{25}(\bibinfo{number}{1}):\bibinfo{pages}{286--297}.
\bibitem[{Bierm\'e et~al.(2007)Bierm\'e, Meerschaert and Scheffler}]{bierme}
\bibinfo{author}{Bierm\'e\xfnm[ H.]}, \bibinfo{author}{Meerschaert\xfnm[ M.]},
  \bibinfo{author}{Scheffler\xfnm[ H.P.]}.
\newblock \bibinfo{title}{Operator scaling stable random fields}.
\newblock \bibinfo{journal}{Stochastic Process Appl}
  \bibinfo{year}{2007};\bibinfo{volume}{117}:\bibinfo{pages}{312--332}.
\bibitem[{Bowman(1984)}]{bowman}
\bibinfo{author}{Bowman\xfnm[ A.W.]}.
\newblock \bibinfo{title}{An alternative method of cross-validation for the
  smoothing of kernel density estimates}.
\newblock \bibinfo{journal}{Biometrika}
  \bibinfo{year}{1984};\bibinfo{volume}{71}:\bibinfo{pages}{353--360}.
\bibitem[{Brockwell and Davis(2016)}]{brockwell}
\bibinfo{author}{Brockwell\xfnm[ P.J.]}, \bibinfo{author}{Davis\xfnm[ R.A.]}.
\newblock \bibinfo{title}{Time Series: Theory and Methods}.
\newblock \bibinfo{publisher}{Springer}, \bibinfo{year}{2016}.
\bibitem[{Cambanis et~al.(1987)Cambanis, Hardin and Weron}]{cambanis1}
\bibinfo{author}{Cambanis\xfnm[ S.]}, \bibinfo{author}{Hardin Jr.\xfnm[ C.D.]},
  \bibinfo{author}{Weron\xfnm[ A.]}.
\newblock \bibinfo{title}{Ergodic properties of stationary stable processes}.
\newblock \bibinfo{journal}{Stochastic Process Appl} \bibinfo{year}{1987};.
\bibitem[{Cambanis et~al.(1992)Cambanis, Maejima and
  Samorodnitsky}]{charfracstable}
\bibinfo{author}{Cambanis\xfnm[ S.]}, \bibinfo{author}{Maejima\xfnm[ M.]},
  \bibinfo{author}{Samorodnitsky\xfnm[ G.]}.
\newblock \bibinfo{title}{Characterization of linear and harmonizable
  fractional stable motions}.
\newblock \bibinfo{journal}{Stochastic Process Appl}
  \bibinfo{year}{1992};\bibinfo{volume}{42}:\bibinfo{pages}{91--110}.
\bibitem[{Doob(1991)}]{doob}
\bibinfo{author}{Doob\xfnm[ J.L.]}.
\newblock \bibinfo{title}{Stochastic Processes}.
\newblock \bibinfo{publisher}{Wiley}, \bibinfo{year}{1991}.
\bibitem[{Dym and McKean(1976)}]{dym}
\bibinfo{author}{Dym\xfnm[ H.]}, \bibinfo{author}{McKean\xfnm[ H.P.]}.
\newblock \bibinfo{title}{Gaussian processes, function theory, and the inverse
  spectral problem}.
\newblock \bibinfo{publisher}{Academic Press}, \bibinfo{year}{1976}.
\bibitem[{Einmahl and Mason(2005)}]{einmahl}
\bibinfo{author}{Einmahl\xfnm[ U.]}, \bibinfo{author}{Mason\xfnm[ D.M.]}.
\newblock \bibinfo{title}{Uniform in bandwidth consistency of kernel-type
  function estimators}.
\newblock \bibinfo{journal}{Ann Statist}
  \bibinfo{year}{2005};\bibinfo{volume}{33}:\bibinfo{pages}{1380--1403}.
\bibitem[{Fuller(1996)}]{fuller}
\bibinfo{author}{Fuller\xfnm[ W.A.]}.
\newblock \bibinfo{title}{Introduction to Statistical Time Series}.
\newblock \bibinfo{publisher}{Wiley}, \bibinfo{year}{1996}.
\bibitem[{Gin\'{e} and Guillou(2002)}]{gineguillou}
\bibinfo{author}{Gin\'{e}\xfnm[ E.]}, \bibinfo{author}{Guillou\xfnm[ A.]}.
\newblock \bibinfo{title}{Rates of strong uniform consistency for multivariate
  kernel density estimators}, \bibinfo{year}{2002}.
\bibitem[{Hannan(1973)}]{hannan2}
\bibinfo{author}{Hannan\xfnm[ E.J.]}.
\newblock \bibinfo{title}{The estimation of frequency}.
\newblock \bibinfo{journal}{J Appl Probab}
  \bibinfo{year}{1973};\bibinfo{volume}{10}:\bibinfo{pages}{510--519}.
\bibitem[{Hoang(2023)}]{code}
\bibinfo{author}{Hoang\xfnm[ L.V.]}.
\newblock \bibinfo{title}{Matlab and r implementation for the spectral density
  estimation fo stationary real harmonizable symmetric $\alpha$-stable
  processes}.
\newblock \bibinfo{journal}{\url{https://githubcom/lyviho/harmonizablestable}}
  \bibinfo{year}{2023};.
\bibitem[{Hoang and Spodarev(2021)}]{viet}
\bibinfo{author}{Hoang\xfnm[ L.V.]}, \bibinfo{author}{Spodarev\xfnm[ E.]}.
\newblock \bibinfo{title}{Inversion of $\alpha$-sine and $\alpha$-cosine
  transforms on $\mathbb{R}$}.
\newblock \bibinfo{journal}{Inverse Problems}
  \bibinfo{year}{2021};\bibinfo{volume}{37}.
\bibitem[{Kallenberg(2002)}]{kallenberg}
\bibinfo{author}{Kallenberg\xfnm[ O.]}.
\newblock \bibinfo{title}{Foundations of Modern Probability}.
\newblock \bibinfo{publisher}{Springer}, \bibinfo{year}{2002}.
\bibitem[{Koutrouvelis(1980)}]{koutrouvelis}
\bibinfo{author}{Koutrouvelis\xfnm[ I.A.]}.
\newblock \bibinfo{title}{Regression-type estimation of the parameters of
  stable laws}.
\newblock \bibinfo{journal}{J Amer Statist Assoc}
  \bibinfo{year}{1980};\bibinfo{volume}{75}:\bibinfo{pages}{918--928}.
\bibitem[{Lamperti(1962)}]{lamperti}
\bibinfo{author}{Lamperti\xfnm[ J.]}.
\newblock \bibinfo{title}{Semi-stable stochastic processes}.
\newblock \bibinfo{journal}{Trans Amer Math Soc}
  \bibinfo{year}{1962};\bibinfo{volume}{104}:\bibinfo{pages}{62--78}.
\bibitem[{Lindgren(2012)}]{lindgren}
\bibinfo{author}{Lindgren\xfnm[ G.]}.
\newblock \bibinfo{title}{Stationary Stochastic Processes: Theory and
  Applications}.
\newblock \bibinfo{publisher}{Chapmann and Hall}, \bibinfo{year}{2012}.
\bibitem[{McCulloch(1986)}]{mcculloch}
\bibinfo{author}{McCulloch\xfnm[ J.H.]}.
\newblock \bibinfo{title}{Simple consistent estimators of stable distribution
  parameters}.
\newblock \bibinfo{journal}{Comm Statist Simulation Comput}
  \bibinfo{year}{1986};\bibinfo{volume}{15}:\bibinfo{pages}{1109--1136}.
\bibitem[{Oppenheim et~al.(1999)Oppenheim, Buck and Schafer}]{signalprocessing}
\bibinfo{author}{Oppenheim\xfnm[ A.V.]}, \bibinfo{author}{Buck\xfnm[ J.R.]},
  \bibinfo{author}{Schafer\xfnm[ R.W.]}.
\newblock \bibinfo{title}{Discrete-time Signal Processing}.
\newblock \bibinfo{publisher}{Prentice Hall}, \bibinfo{year}{1999}.
\bibitem[{Quinn and Hannan(2001)}]{hannan}
\bibinfo{author}{Quinn\xfnm[ B.G.]}, \bibinfo{author}{Hannan\xfnm[ E.J.]}.
\newblock \bibinfo{title}{The Estimation and Tracking of Frequency}.
\newblock \bibinfo{publisher}{Cambridge University Press},
  \bibinfo{year}{2001}.
\bibitem[{Rosinski(1995)}]{rosinski}
\bibinfo{author}{Rosinski\xfnm[ J.]}.
\newblock \bibinfo{title}{On the structure of stationary stable processes}.
\newblock \bibinfo{journal}{Ann Probab}
  \bibinfo{year}{1995};\bibinfo{volume}{23}:\bibinfo{pages}{1163--1187}.
\bibitem[{Rudeom(1982)}]{rudemo}
\bibinfo{author}{Rudeom\xfnm[ M.]}.
\newblock \bibinfo{title}{Empirical choice of histograms and kernel density
  estimators}.
\newblock \bibinfo{journal}{Scand J Stat}
  \bibinfo{year}{1982};\bibinfo{volume}{9}:\bibinfo{pages}{65--78}.
\bibitem[{Samorodnitsky and Taqqu(1994)}]{gennady}
\bibinfo{author}{Samorodnitsky\xfnm[ G.]}, \bibinfo{author}{Taqqu\xfnm[ M.S.]}.
\newblock \bibinfo{title}{Stable Non-Gaussian Random Processes: Stochastic
  Models with Infinite Variance}.
\newblock \bibinfo{publisher}{Chapman and Hall}, \bibinfo{year}{1994}.
\bibitem[{Sheather(2004)}]{sheather2}
\bibinfo{author}{Sheather\xfnm[ S.J.]}.
\newblock \bibinfo{title}{Density estimation}.
\newblock \bibinfo{journal}{Statist Sci}
  \bibinfo{year}{2004};\bibinfo{volume}{19}:\bibinfo{pages}{588--597}.
\bibitem[{Sheather and Jones(1984)}]{sheather}
\bibinfo{author}{Sheather\xfnm[ S.J.]}, \bibinfo{author}{Jones\xfnm[ M.C.]}.
\newblock \bibinfo{title}{A reliable data-based bandwidth selection method for
  kernel density estimation}.
\newblock \bibinfo{journal}{J R Stat Soc Ser B Stat Methodol}
  \bibinfo{year}{1984};\bibinfo{volume}{53}:\bibinfo{pages}{683--690}.
\bibitem[{Silverman(1986)}]{silverman}
\bibinfo{author}{Silverman\xfnm[ B.W.]}.
\newblock \bibinfo{title}{Density Estimation for Statistics and Data Analysis}.
\newblock \bibinfo{publisher}{Chapman and Hall}, \bibinfo{year}{1986}.
\bibitem[{\'Slezak(2017)}]{slezak}
\bibinfo{author}{\'Slezak\xfnm[ J.]}.
\newblock \bibinfo{title}{Asymptotic behaviour of time averages for non-ergodic
  gaussian processes}.
\newblock \bibinfo{journal}{Ann Physics}
  \bibinfo{year}{2017};\bibinfo{volume}{383}:\bibinfo{pages}{285--311}.
\bibitem[{Walker(1971)}]{walker}
\bibinfo{author}{Walker\xfnm[ A.M.]}.
\newblock \bibinfo{title}{On the estimation of a harmonic component in a time
  series with stationary independent residuals}.
\newblock \bibinfo{journal}{Biometrika}
  \bibinfo{year}{1971};\bibinfo{volume}{58}:\bibinfo{pages}{21--36}.
\bibitem[{Wied and Weißbach(2012)}]{kdeconsistency}
\bibinfo{author}{Wied\xfnm[ D.]}, \bibinfo{author}{Weißbach\xfnm[ R.]}.
\newblock \bibinfo{title}{Consistency of the kernel density estimator - a
  survey}.
\newblock \bibinfo{journal}{Statist Papers}
  \bibinfo{year}{2012};\bibinfo{volume}{53}:\bibinfo{pages}{1--21}.
\bibitem[{Xiao and Ayache(2016)}]{xiao}
\bibinfo{author}{Xiao\xfnm[ Y.]}, \bibinfo{author}{Ayache\xfnm[ A.]}.
\newblock \bibinfo{title}{Harmonizable fractional stable fields: local
  nondeterminism and joint continuity of the local times}.
\newblock \bibinfo{journal}{Stochastic Process Appl}
  \bibinfo{year}{2016};\bibinfo{volume}{126}:\bibinfo{pages}{117--185}.

\end{thebibliography}

\section{Periodogram computation}
\label{appendixA}

Recall the discrete Fourier transform of the sample $x(1),\dots,x(n)$ for $\theta>0$ in Equation (11), i.e.
\begin{align*}
	F_n(\theta\delta)&
	=n^{-1}\sum\limits_{l=1}^\infty\Bigg( \frac{e^{i\Theta_l}R_l}{2}\underbrace{\sum\limits_{j=1}^ne^{ij\delta\left(Z_l-\theta\right)}}_{=S_l^{(1)}(\theta,\delta)}
	+\frac{e^{-i\Theta_l}R_l}{2}\underbrace{\sum\limits_{j=1}^ne^{-ij\delta\left(Z_l+\theta\right)}}_{=S_l^{(2)}(\theta,\delta)}\Bigg),
\end{align*}
where the sums $S_l^{(1)}$ and $S_l^{(2)}$ can be explicitly expressed by 
\begin{align*}
	S_l^{(1)}(\theta,\delta)=e^{i\delta\left(Z_l-\theta\right)}\frac{e^{in\delta\left(Z_l-\theta\right)}-1}{e^{i\delta\left(Z_l-\theta\right)}-1}\quad\text{and}\quad S_l^{(2)}(\theta,\delta)=e^{-in\delta\left(Z_l-\theta\right)}\frac{e^{in\delta\left(Z_l+\theta\right)}-1}{e^{i\delta\left(Z_l+\theta\right)}-1}.
\end{align*}

The periodogram $I_n\left(\theta\right)$ is defined as the squared absolute value of the discrete Fourier transform. Note that $\left\vert x+y\right\vert^2 = \vert x\vert ^2+\vert y \vert ^2+x^\ast y + xy^\ast=\vert x\vert ^2+\vert y \vert ^2+2Re(x^\ast y)$, where $x^\ast$ denotes the complex conjugate of $x$.
This yields
\begin{align*}
	I\left(\theta\right)&=\Bigg\vert n^{-1}\sum\limits_{l=1}^\infty \Bigg( \frac{e^{i\Theta_l}R_l}{2}S_l^{(1)}(\theta,\delta)
	+\frac{e^{-i\Theta_l}R_l}{2}S_l^{(2)}(\theta,\delta)\Bigg)\Bigg\vert^2	\\
	&=n^{-2}\sum\limits_{l=1}^\infty\frac{R_l^2}{4}\left[\left\vert S_l^{(1)}(\theta,\delta)\right\vert^2+\left\vert S_l^{(2)}(\theta,\delta)\right\vert^2 +2 Re\left(\left(e^{i\Theta_l}S_l^{(1)}(\theta,\delta)\right)^\ast e^{-i\Theta_l}S_l^{(2)}(\theta,\delta)\right)\right]\\
	&+n^{-2}\sum\limits_{l=1}^\infty\sum\limits_{m>l}\frac{R_lR_m}{2}Re\left(\left(e^{i\Theta_l}S_l^{(1)}(\theta,\delta)+e^{-i\Theta_l}S_l^{(2)}(\theta,\delta)\right)^\ast\left(e^{i\Theta_m}S_m^{(1)}(\theta,\delta)+e^{-i\Theta_m}S_m^{(2)}(\theta,\delta)\right)\right)\\
	&=(\ast_1)+(\ast_2).
\end{align*}

For ease of notation we assume $\delta=1$ and define the function $S_n(x)=\sin\left(\frac{n}{2}x\right)/\sin\left(\frac{1}{2}x\right)$.
In the following we will not further denote the functions dependence on $n\in\mathbb{N}$ with a subscript and just write $I=I_n$ and $S=S_n$ instead.
The function $S$ is well-defined for all $x\in\mathbb{R}$, in particular it holds that $S(0)=n$, hence it follows that $\lim_{n\rightarrow\infty}n^{-1}S(x)=\pm 1$ for any $x= j\pi$, $j\in\mathbb{Z}$, and $0$ otherwise. 

Note that $\left\vert e^{ix}-1\right\vert^2=\left(\cos(x)-1\right)^2+\sin^2(x)=2-2\cos(x)=4\sin^2\left(\frac{x}{2}\right)$ for all $x\in\mathbb{R}$, hence we can compute 
\begin{align*}
	\left\vert S_l^{(1)}(\theta,1)\right\vert^2=\frac{\sin^2\left(\frac{n}{2}(Z_l-\theta)\right)}{\sin^2\left(\frac{1}{2}(Z_l-\theta)\right)}=S^2\left(Z_l-\theta\right),
	\quad
	\left\vert S_l^{(2)}(\theta,1)\right\vert^2 = \frac{\sin^2\left(\frac{n}{2}(Z_l+\theta)\right)}{\sin^2\left(\frac{1}{2}(Z_l+\theta)\right)}=S^2\left(Z_l+\theta\right)
\end{align*}
for all $l\in\mathbb{N}$.
Furthermore, it holds that  
\begin{align*}
	\left(e^{i\Theta_l}S_l^{(1)}(\theta,1)\right)^\ast \left(e^{-i\Theta_l}S_l^{(2)}(\theta,1)\right)=e^{-i2\Theta_l}\left(S_l^{(1)}\left(\theta,1\right)\right)^\ast S_l^{(2)}\left(\theta,1\right)
\end{align*}
with
\begin{align*}
	\left(S_l^{(1)}\left(\theta,1\right)\right)^\ast S_l^{(2)}\left(\theta,1\right)
	=S(Z_l-\theta)S(Z_l+\theta)\cos\left((n+1)Z_l\right)-iS(Z_l-\theta)S(Z_l+\theta)\sin\left((n+1)Z_l\right).
\end{align*}
Hence,
\begin{align}
	\label{mixedprodS1S2}
	&Re\left(e^{-i2\Theta_l}\left(S_l^{(1)}\left(\theta,1\right)\right)^\ast S_l^{(2)}\left(\theta,1\right)\right)\nonumber\\
	=&Re\left(e^{-i2\Theta_l}\right)Re\left(\left(S_l^{(1)}\left(\theta,1\right)\right)^\ast S_l^{(2)}\left(\theta,1\right)\right)-Im\left(e^{-i2\Theta_l}\right)Im\left(\left(S_l^{(1)}\left(\theta,1\right)\right)^\ast S_l^{(2)}\left(\theta,1\right)\right)\nonumber\\
	=&S(Z_l-\theta)S(Z_l+\theta)\left(\cos(2\Theta_l)\cos\left((n+1)Z_l\right)-\sin(2\Theta_l)\sin\left((n+1)Z_l\right)\right)\nonumber\\
	=&S(Z_l-\theta)S(Z_l+\theta)\cos\left(2\Theta_l+(n+1)Z_l\right).
\end{align}
Therefore, 
\begin{align*}
	(\ast_1)= n^{-2}\sum\limits_{l=1}^\infty \frac{R_l^2}{4} \left[S^2(Z_l-\theta)+S^2(Z_l+\theta)+2\cos\left(2\Theta_l+(n+1)Z_l\right)S(Z_l-\theta)S(Z_l+\theta)\right].
\end{align*}
For the second sum $(\ast_2)$ first compute 
\begin{align*}
	&\left(e^{i\Theta_l}S_l^{(1)}(\theta,\delta)+e^{-i\Theta_l}S_l^{(2)}(\theta,\delta)\right)^\ast\left(e^{i\Theta_m}S_m^{(1)}(\theta,\delta)+e^{-i\Theta_m}S_m^{(2)}(\theta,\delta)\right)\\
	=&e^{i(-\Theta_l+\Theta_m)}\left(S_l^{(1)}\left(\theta,1\right)\right)^\ast S_m^{(1)}\left(\theta,1\right)
	+e^{i(-\Theta_l-\Theta_m)}\left(S_l^{(1)}\left(\theta,1\right)\right)^\ast S_m^{(2)}\left(\theta,1\right)	\\
	&+e^{i(\Theta_l+\Theta_m)}\left(S_l^{(2)}\left(\theta,1\right)\right)^\ast S_l^{(1)}\left(\theta,1\right)
	+e^{i(\Theta_l-\Theta_m)}\left(S_l^{(2)}\left(\theta,1\right)\right)^\ast S_m^{(2)}\left(\theta,1\right).
\end{align*}
The real part of the above sum is the sum of the real part of each summand, and similar to $(\ast1)$ we can compute
\begin{align*}
	&Re\left(\left(e^{i\Theta_l}S_l^{(1)}(\theta,\delta)+e^{-i\Theta_l}S_l^{(2)}(\theta,\delta)\right)^\ast\left(e^{i\Theta_m}S_m^{(1)}(\theta,\delta)+e^{-i\Theta_m}S_m^{(2)}(\theta,\delta)\right)\right)\\
	=&S(Z_l-\theta)S(Z_m-\theta)\cos\left(-\Theta_l+\Theta_m+\frac{1}{2}(n+1)(-Z_l+Z_m)\right)\\
	&+S(Z_l-\theta)S(Z_m+\theta)\cos\left(-\Theta_l-\Theta_m+\frac{1}{2}(n+1)(-Z_l-Z_m)\right)\\
	&+S(Z_l+\theta)S(Z_m-\theta)\cos\left(\Theta_l+\Theta_m+\frac{1}{2}(n+1)(Z_l+Z_m)\right)\\
	&+S(Z_l+\theta)S(Z_m+\theta)\cos\left(\Theta_l-\Theta_m+\frac{1}{2}(n+1)(Z_l-Z_m)\right)\\		=&\left(S(Z_l-\theta)S(Z_m-\theta)+S(Z_l+\theta)S(Z_m+\theta)\right)\cos\left(\Theta_l-\Theta_m+\frac{1}{2}(n+1)(Z_l-Z_m)\right)\\
	&+\left(S(Z_l-\theta)S(Z_m+\theta)+S(Z_l+\theta)S(Z_m-\theta)\right)\cos\left(\Theta_l+\Theta_m+\frac{1}{2}(n+1)(Z_l+Z_m)\right).
\end{align*}
It follows that 
\begin{align*}
	(\ast_2)&=n^{-2}\sum\limits_{l=1}^\infty\sum\limits_{m>l}\frac{R_lR_m}{2}\\
	&\times\left[\left(S(Z_l-\theta)S(Z_m-\theta)+S(Z_l+\theta)S(Z_m+\theta)\right)\cos\left(\Theta_l-\Theta_m+\frac{1}{2}(n+1)(Z_l-Z_m)\right)\right.\\
	&\left.+\left(S(Z_l-\theta)S(Z_m+\theta)+S(Z_l+\theta)S(Z_m-\theta)\right)\cos\left(\Theta_l+\Theta_m+\frac{1}{2}(n+1)(Z_l+Z_m)\right)\right].
\end{align*}
Note that one can easily verify that the frequencies $Z_l$ and $Z_m$ inside the function $S$ can be replaced by their absolute values $\vert Z_l\vert$ and $\vert Z_m\vert$.
Ultimately, we get
\begin{align}
	\label{periodogramexact}
	I(\theta)&=n^{-2}\sum\limits_{l=1}^\infty \frac{R_l^2}{4}\left[S^2(\vert Z_l\vert-\theta)+S^2(\vert Z_l\vert+\theta)+2\cos\left(2\Theta_l+(n+1) Z_l\right)S(\vert Z_l\vert-\theta)S(\vert Z_l\vert+\theta)\right]\nonumber\\
	&+n^{-2}\sum\limits_{l=1}^\infty\sum\limits_{m>l}\frac{R_lR_m}{2}\nonumber\\
	&\times\left[\left(S(\vert Z_l\vert-\theta)S(\vert Z_m\vert-\theta)+S(\vert Z_l\vert+\theta)S(\vert Z_m\vert+\theta)\right)\cos\left(\Theta_l-\Theta_m+\frac{1}{2}(n+1)( Z_l-Z_m)\right)\right.\nonumber\\
	&+\left.\left(S(\vert Z_l\vert-\theta)S(\vert Z_m\vert+\theta)+S(\vert Z_l\vert+\theta)S(\vert Z_m\vert-\theta)\right)\cos\left(\Theta_l+\Theta_m+\frac{1}{2}(n+1)(Z_l+Z_m)\right)\right].
\end{align}

The periodogram's convergence to $0$ for all $\theta\not\in\{Z_k\}_{k\in\mathbb{N}}$ with rate $O(n^{-2})$ as $n\rightarrow\infty$, except for when $\theta\rightarrow\vert Z_k\vert$ for some $k\in\mathbb{N}$, can be readily established by Lebesgue's dominated convergence theorem and the fact that $n^{-1}\vert S(x)\vert$ is bounded by $1$.

\section{Computations in the proof of Theorem 4}
\label{appendixB}

\vspace{3ex}
\begin{enumerate}[1.]
	
	\item \emph{Show that $\hat{Z}_{1,n}\rightarrow \vert Z_{[1]}\vert $ a.s. as $n\rightarrow\infty.$}
	
	\begin{enumerate}[(i)]
		\item Assume that $\hat{Z}_{1,n}(\omega)\rightarrow\vert Z_{[j]}(\omega)\vert$ as $n\rightarrow\infty$ for some $j>1$, $\omega\in\Omega$. Since $\hat{Z}_{1,n}(\omega)$ is the location of the largest peak of the periodogram $I_n(\theta;\omega)$ it follows that 
		\begin{align*}
			\frac{R^2_{[j]}(\omega)}{4}\leftarrow I_n(\hat{Z}_{1,n}(\omega);\omega)\geq I_n(\vert Z_{[1]}(\omega)\vert;\omega)\rightarrow\frac{R^2_{[1]}(\omega)}{4},
		\end{align*}
		which is a contradiction  since $R_{[1]}(\omega)>R_{[j]}(\omega)$ for all $j>1$ and almost all $\omega\in\Omega$. 
		
		\item Assume that $\hat{Z}_{1,n}(\omega)\rightarrow z'\not\in\left\{\vert Z_{[j]}(\omega)\vert\right\}_{j\in\mathbb{N}}$ as $n\rightarrow\infty$ and $\omega\in\Omega$. Then, 
		\begin{align*}
			0\leftarrow I_n(\hat{Z}_{1,n}(\omega);\omega)\geq I_n(\vert Z_{[1]}(\omega)\vert;\omega)\rightarrow\frac{R^2_{[1]}(\omega)}{4},
		\end{align*}
		which is again a contradiction, since $R_{[1]}(\omega)>0$ for almost all $\omega\in\Omega$.
		
		\item Assume $\hat{Z}_{1,n}(\omega)$ does not converge. Then, there exists a converging subsequence $(\hat{Z}_{1,n_l}(\omega))_{l\in\mathbb{N}}$ (since $(\hat{Z}_{1,n}(\omega))_{n\in\mathbb{N}}$ is bounded) with $\hat{Z}_{1,n_l}(\omega)\rightarrow z'\neq \vert Z_{[1]}(\omega)\vert$
		\begin{align*}
			\left.\begin{aligned}
				z'\not\in\{\vert Z_{[j]}\},\quad&0\\
				z'=\vert Z_{[j]}\vert, \ j>1,\quad&R^2_{[j]}/4\\
			\end{aligned}\right\}
			\leftarrow I_{n_l}(\hat{Z}_{1,n_l})\geq I_{n_l}(\vert Z_{[1]}\vert)\rightarrow\frac{R^2_{[1]}}{4}
		\end{align*}
	\end{enumerate}
	Hence, $\hat{Z}_{1,n}\rightarrow \vert Z_{[1]}\vert $ a.s.

	\item \emph{Show that $n(\hat{Z}_{1,n}-\vert Z_{[1]}\vert)\rightarrow 0$ a.s. as $n\rightarrow\infty.$}
	
	
	Follows from $\hat{Z}_{1,n}\rightarrow\vert Z_{[1]}\vert$ a.s. as $n\rightarrow\infty$ since for any $\omega\in\Omega$
	\begin{align*}
		0&\leq I_n(\hat{Z}_{1,n}(\omega);\omega)-I_n(\vert Z_{[1]}(\omega)\vert;\omega)\\
		&=\frac{R_{[1]}^2(\omega)}{4}\underbrace{\Bigg(\underbrace{\frac{\sin^2\left(\frac{1}{2}n(\vert Z_{[1]}(\omega)\vert-\hat{Z}_{1,n}(\omega))\right)}{n^2\sin^2\left(\frac{1}{2}(\vert Z_{[1]}(\omega)\vert-\hat{Z}_{1,n}(\omega))\right)}}_{\leq 1}-1\Bigg)}_{\leq 0}+O\left(n^{-2}\right).
	\end{align*}
	Since the first term in the above is non-positive and the second term vanishes as $n\rightarrow\infty$ but the whole expression is non-negative, it follows that 
	\begin{align*}
		\frac{R_{[1]}^2(\omega)}{4}\Bigg(\frac{\sin^2\left(\frac{1}{2}n(\vert Z_{[1]}(\omega)\vert-\hat{Z}_{1,n}(\omega))\right)}{n^2\sin^2\left(\frac{1}{2}(\vert Z_{[1]}(\omega)\vert-\hat{Z}_{1,n}(\omega))\right)}-1\Bigg)\rightarrow 0 
	\end{align*}
	as $n\rightarrow\infty$ for any $\omega\in\Omega$. Consequently 
	\begin{align*}
		\frac{\sin^2\left(\frac{1}{2}n(\vert Z_{[1]}(\omega)\vert-\hat{Z}_{1,n}(\omega))\right)}{n^2\sin^2\left(\frac{1}{2}(\vert Z_{[1]}(\omega)\vert-\hat{Z}_{1,n}(\omega))\right)}\rightarrow 1
	\end{align*}
	as $n\rightarrow\infty$, which is equivalent to $n(\hat{Z}_{1,n}-\vert Z_{[1]}\vert)\rightarrow 0$ as $n\rightarrow\infty$ for any $\omega\in\Omega$. 
	
	\item \emph{Show that the least-squares estimators $(\hat{a}_{1,n},\hat{b}_{1,n})$ from the regression of $a_1\cos(t\hat{Z}_{1,n})+b_1\sin(t\hat{Z}_1)$, $t=1,\dots,n$, on $x(1),\dots,x(n)$ converge a.s., i.e. 
		\begin{align*}
			\hat{a}_{1,n}\rightarrow a_{[1]},\qquad\hat{b}_{1,n}\rightarrow b_{[1]} \quad a.s.
		\end{align*}
		as $n\rightarrow\infty$, where $a_{[k]},b_{[k]}$ are the coefficients in the cosine-sine series representation of $X$, ordered according to the order of the descending $R_{[k]}$, i.e.
		\begin{align*}
			X(t)=\sum_{k=1}^\infty a_{[k]}\cos(Z_{[k]}t)+b_{[k]}\sin(Z_{[k]}t)=\sum_{k=1}^\infty A_{[k]}e^{iZ_{[k]}t}+A_{[k]}^\ast e^{-iZ_{[k]}t}
		\end{align*}
		with
		\begin{align*}
			a_{[k]}=C_\alpha\Gamma^{-1/\alpha}_{[k]}G^{(1)}_{[k]},\quad b_{[k]}= C_\alpha\Gamma^{-1/\alpha}_{[k]}G^{(2)}_{[k]},\quad A_{[k]}=\frac{a_{[k]}-ib_{[k]}}{2}.
		\end{align*}
		For details see Paper, Proof of Theorem 2, in particular Equations (8) and (9).
	}
	
	Proof according to \cite[Section 3]{walker}. 
	
	The explicit forms of $\hat{a}$ and $\hat{b}$ are given by 
	\begin{align*}
		\hat{a}_{1,n}=\frac{2}{n}\sum_{t=1}^nx(t)\cos(\hat{Z}_{1,n}t),\quad\hat{b}_{1,n}=\frac{2}{n}\sum_{t=1}^nx(t)\sin(\hat{Z}_{1,n}t).
	\end{align*}
	We can write
	\begin{align*}
		\hat{a}_{1,n}+i\hat{b}_{1,n}&=\frac{2}{n}\sum_{t=1}^nx(t)\left(\cos(\hat{Z}_{1,n}t)+i\sin(\hat{Z}_{1,n}t)\right)\\
		&=\frac{2}{n}\sum_{t=1}^nx(t)e^{i\hat{Z}_{1,n}t}\\
		&=\frac{2}{n}\sum_{t=1}^n\left(\sum_{k=1}^\infty a_{[k]}\cos(Z_{[k]}t)+b_{[k]}\sin(Z_{[k]}t)\right)e^{i\hat{Z}_{1,n}t}\\
		&=\frac{2}{n}\sum_{t=1}^n\left(a_{[1]}\cos(Z_{[k]}t)+b_{[1]}\sin(Z_{[1]}t)\right)e^{i\hat{Z}_{1,n}t}\\
		& \ \ +\frac{2}{n}\sum_{t=1}^n\left(\sum_{k=2}^\infty a_{[k]}\cos(Z_{[k]}t)+b_{[k]}\sin(Z_{[k]}t)\right)e^{i\hat{Z}_{1,n}t}\\
		&=\frac{2}{n}\sum_{t=1}^n\left(A_{[1]}e^{iZ_{[1]}t}+A^{\ast}_{[1]}e^{-iZ_{[1]}t}\right)e^{i\hat{Z}_{1,n}t}\\
		& \ \ +\frac{2}{n}\sum_{t=1}^n\left(\sum_{k=2}^\infty A_{[k]}e^{iZ_{[k]}t}+A^{\ast}_{[k]}e^{-iZ_{[k]}t}\right)e^{i\hat{Z}_{1,n}t}\\
		&=\frac{2}{n}\sum_{t=1}^n\left(A_{[1]}e^{i(\hat{Z}_{1,n}+Z_{[1]})t}+A^{\ast}_{[1]}e^{i(\hat{Z}_{1,n}-Z_{[1]})t}\right) + \frac{2}{n}\sum_{t=1}^n\left(\sum_{k=2}^\infty \dots\right)e^{i\hat{Z}_{1,n}t}
	\end{align*}
	It follows that 
	\begin{align*}
		&\left(\hat{a}_{1,n}-a_{[1]}\right)+i\left(\hat{b}_{1,n}-b_{[1]}\right)\\
		&=\hat{a}_{1,n}+i\hat{b}_{1,n}-\left(a_{[1]}+ib_{[1]}\right)\\
		&=\frac{2}{n}\sum_{t=1}^n \Bigg(x(t)e^{i\hat{Z}_{1,n}}-\underbrace{\frac{a_{[1]}+ib_{[1]}}{2}}_{=A^\ast_{[1]}}\Bigg)\\
		&=\frac{2}{n}\sum_{t=1}^n \left(A_{[1]}e^{i(\hat{Z}_{1,n}+Z_{[1]})t}+A^{\ast}_{[1]}e^{i(\hat{Z}_{1,n}-Z_{[1]})t}-A^\ast_{[1]}\right) +\frac{2}{n}\sum_{t=1}^n\left(\sum_{k=2}^\infty \dots\right)e^{i\hat{Z}_{1,n}t}.
	\end{align*}
	Recall that $$M_n(x):=\sum_{t=1}^ne^{ixt}=e^{ix}\frac{e^{inx}-1}{e^{ix}-1}
	,$$
	hence by the triangle inequality
	\begin{align*}
		&\left\vert\left(\hat{a}_{1,n}-a_{[1]}\right)+i\left(\hat{b}_{1,n}-b_{[1]}\right)\right\vert\\
		&\leq \frac{2}{n}\left\vert A_{[1]}\right\vert\left(\left\vert M_n(\hat{Z}_{1,n}+Z_{[1]})\right\vert+\left\vert M_n(\hat{Z}_{1,n}-Z_{[1]})-n\right\vert\right)+\frac{2}{n}\left\vert\sum_{t=1}^n\left(\sum_{k=2}^\infty \dots\right)e^{i\hat{Z}_{1,n}t}\right\vert.
	\end{align*}
	We can replace $Z_{[1]}$ by $\vert Z_{[1]}\vert$ in the above (for the case $Z_{[1]}<0$ simply consider $\left\vert M_n(\hat{Z}_{1,n}-Z_{[1]})\right\vert+\left\vert M_n(\hat{Z}_{1,n}+Z_{[1]})-n\right\vert$ instead). 
	Then, first of all
	\begin{align*}
		n^{-1}\left\vert M_n(\hat{Z}_{1,n}+\vert Z_{[1]}\vert)\right\vert=\frac{\sin\left(\frac{n}{2}(\hat{Z}_{1,n}+\vert Z_{[1]})\right)}{n\sin\left(\frac{1}{2}(\hat{Z}_{1,n}+\vert Z_{[1]})\right)}\rightarrow 0\quad a.s.
	\end{align*}
	as $n\rightarrow\infty$. As for the second term, applying the mean value theorem and the fact that $\left\vert M_n'(x)\right\vert=\left\vert\sum_{t=1}^nte^{ixt}\right\vert < n^2$ for all $x\in\mathbb{R}$ \cite[p. 27]{walker} yields
	\begin{align*}
		n^{-1}\left\vert M_n(\hat{Z}_{1,n}-\vert Z_{[1]}\vert)-n\right\vert = n^{-1}\left\vert M_n(\hat{Z}_{1,n}-\vert Z_{[1]}\vert)-M_n(0)\right\vert < \text{const} \ n\left\vert \hat{Z}_{1,n}-\vert Z_{[1]}\vert\right\vert \rightarrow 0 \quad a.s.
	\end{align*}
	as $n\rightarrow\infty$. For the third term, note that 
	\begin{align*}
		\frac{1}{n}\sum_{t=1}^n\sum_{k=2}^{\infty}\left(A_{[k]}e^{iZ_{[k]}t}+A^{\ast}_{[k]}e^{-iZ_{[k]}t}\right)e^{i\hat{Z}_{1,n}t}=\sqrt{\tilde{I}_n(\hat{Z}_{1,n})}
	\end{align*}
	is the square root of the periodogram of a sample of a harmonic process with frequencies $\{Z_{[k]}\}_{k\geq 2}$. Since $\hat{Z}_{1,n}\rightarrow\vert Z_{[1]}\vert$ a.s. and $\vert Z_{[1]}\vert\not\in\{Z_{[k]}\}_{k\geq 2}$ a.s. it follows that $\tilde{I}_n(\hat{Z}_{1,n})\rightarrow 0$ a.s. Consequently, we have $\sqrt{\tilde{I}_n(\hat{Z}_{1,n})}\rightarrow 0$ a.s. by the continuous mapping theorem. 
	
	%
\end{enumerate}

\end{document}